\documentclass[ACS,STIX1COL]{WileyNJD-v2}

\usepackage{latexsym}
\usepackage{amsmath}
\usepackage{amsfonts}
\usepackage{amssymb}
\usepackage{stmaryrd}
\usepackage{dsfont}
\usepackage{color}

\articletype{Original Paper}%

\newcommand{\tr}{{\rm tr}}

\renewcommand{\div}{{\rm div}}

\newcommand{\bcurl}{{\bf curl}}
\renewcommand{\S}{\mathds{S}}
\newcommand{\R}{\mathds{R}}
\newcommand{\N}{\mathds{N}}
\renewcommand{\R}{{\rm I\kern-.25em R}}
\renewcommand{\N}{{\rm I\kern-.25em N}}

\newcommand{\bpsi}{\mbox{\boldmath$\psi$}}

\newcommand{\bvarphi}{\mbox{\boldmath$\varphi$}}

\newcommand{\bPi}{\mbox{\boldmath$\Pi$}}
\newcommand{\bXi}{\mbox{\boldmath$\Xi$}}

\newcommand{\bxi}{\mbox{\boldmath$\xi$}}

\newcommand{\bzeta}{\mbox{\boldmath$\zeta$}}
\newcommand{\bgamma}{\mbox{\boldmath$\gamma$}}

\newcommand{\brho}{\mbox{\boldmath$\rho$}}
\newcommand{\bsigma}{\mbox{\boldmath$\sigma$}}

\newcommand{\btheta}{\mbox{\boldmath$\theta$}}

\newcommand{\bnabla}{\mbox{\boldmath$\nabla$}}
\newcommand{\bzero}{{\bf 0}}
\newcommand{\bff}{{\bf f}}
\newcommand{\bv}{{\bf v}}

\newcommand{\bx}{{\bf x}}
\newcommand{\bz}{{\bf z}}

\newcommand{\bu}{{\bf u}}

\newcommand{\bg}{{\bf g}}
\newcommand{\bq}{{\bf q}}

\newcommand{\bs}{{\bf s}}
\newcommand{\be}{{\bf e}}

\newcommand{\ba}{{\bf a}}

\newcommand{\bB}{{\bf B}}

\newcommand{\bF}{{\bf F}}

\newcommand{\bI}{{\bf I}}
\newcommand{\bJ}{{\bf J}}
\newcommand{\bM}{{\bf M}}

\newcommand{\bP}{{\bf P}}
\newcommand{\bQ}{{\bf Q}}
\newcommand{\bR}{{\bf R}}
\newcommand{\bS}{{\bf S}}

\newcommand{\bV}{{\bf V}}

\newcommand{\bX}{{\bf X}}
\newcommand{\bZ}{{\bf Z}}
\newcommand{\bn}{{\bf n}}
\newcommand{\bt}{{\bf t}}

\newcommand{\cI}{{\cal I}}

\newcommand{\cL}{{\cal L}}

\newcommand{\cP}{{\cal P}}

\newcommand{\cS}{{\cal S}}
\newcommand{\cT}{{\cal T}}
\newcommand{\cV}{{\cal V}}

\received{28 February 2019}
\revised{***}
\accepted{***}

\begin{document}

\title{Weakly symmetric stress equilibration for hyperelastic material models%
\protect\thanks{The authors gratefully acknowledge support by the German Research Foundation (DFG) in the Priority Programm SPP 1748
`Reliable simulation techniques in solid mechanics' under grant numbers BE6511/1-1 and STA 402/14-1.}}

\author[1]{Fleurianne Bertrand}

\author[2]{Marcel Moldenhauer}

\author[2]{Gerhard Starke*}

\authormark{F. Bertrand, M. Moldenhauer, and G. Starke}

\address[1]{\orgdiv{%
Institut f\"ur Mathematik}, \orgname{Humboldt-Universit\"at zu Berlin}, \orgaddress{Unter den Linden 6, 10099 Berlin, \country{Germany}.\\
\email{fb@math.hu-berlin.de}}}

\address[2]{\orgdiv{%
Fakult\"at f\"ur Mathematik}, \orgname{Universit\"at Duisburg-Essen}, \orgaddress{Thea-Leymann-Str. 9, 45127 Essen, \country{Germany}.\\
\email{marcel.moldenhauer@uni-due.de}}}

\corres{*Gerhard Starke, Fakult\"at f\"ur Mathematik, Universit\"at Duisburg-Essen, Thea-Leymann-Str. 9, 45127 Essen, Germany.
\email{gerhard.starke@uni-due.de}}


\abstract{A stress equilibration procedure for hyperelastic material models is proposed and analyzed in this paper. Based on the
  displacement-pressure approximation computed with a stable finite element pair, it constructs, in a vertex-patch-wise manner, an
  $H (\div)$-conforming approximation to the first Piola-Kirchhoff stress. This is done in such a way that its associated Cauchy stress
  is weakly symmetric in the sense that its anti-symmetric part is zero tested against continuous piecewise linear functions. Our main
  result is the identification of the subspace of test functions perpendicular to the range of the local equilibration system on each patch
  which turn out to be rigid body modes associated with the current configuration. Momentum balance properties are investigated
  analytically and numerically and the resulting stress reconstruction is shown to provide improved results for surface traction forces
  by computational experiments.}



\maketitle


\section{Introduction}

\label{sec-introduction}

This paper is concerned with a stress equilibration procedure for hyperelastic material models in nonlinear solid mechanics. It extends the
approach proposed and studied in our earlier work \cite{BerKobMolSta:19} to the case of geometrically and materially nonlinear elasticity in the
form of a hyperelastic
material law. Due to the fact that the symmetry condition does not hold for the first Piola-Kirchhoff stress (which is the result from the reconstruction
process) but for the Cauchy stress, the use of symmetric stress elements is not feasible anymore in the hyperelastic case. The weak symmetry
condition from linear elasticity can, however, be generalized to a suitable constraint for the Piola-Kirchhoff stress as is done in this
contribution. To the best of our knowledge, our contribution is the first attempt to develop a stress equilibration procedure for the hyperelastic
situation. Our hope is that this will be of use for the development of an a posteriori error estimator for hyperelastic problems in the future. The
issue of a posteriori error estimation and adaptive refinement is, however, beyond the scope of this contribution.

Expressing the internal forces of a material, the components of the stress-tensor are crucial for the prediction of the weakening of a material,
including plastic behavior or damage. A specific application area where this is an issue is associated with implant shape design which constitutes
an optimal control problem, see \cite{LubSchWei:14}.
Therefore, the accurate approximation of the stress-tensor is of strong importance in numerous applications
and in particular in the hyperelastic material model this paper is concerned with. The mathematical foundations of hyperelastic material models in
solid mechanics are covered, e.g., in the books by Marsden and Hughes \cite{MarHug:83} and Ciarlet \cite{Cia:88}. The numerical treatment of
the associated variational problems are investigated in detail by Le Tallec \cite{LeT:94}. Specifically for incompressible hyperelasticity, issues
connected to the use of displacement-pressure formulations are discussed in \cite{AurBeiLovReaTayWri:13}. A priori analysis of numerical
methods are available under restrictive assumptions, see Carstensen and Dolzmann \cite{CarDol:04} and, for a least-squares finite element
approach, M\"uller et.al. \cite{MueStaSchSch:14}.

Common displacement-based approaches or, in the incompressible regime, mixed displacement-pressure formulations for this model lead to
approximations of the stresses that are not $H (\div)$-conforming, i.e., have discontinuities of the normal components on the interface between
two elements. In particular, this means on the one hand that they do not control momentum conservation and on the other hand that the normal
component of the boundary traces are not well-defined implying that the approximation of the surface traction forces can also not be guaranteed.
In contrast to variational principles involving a direct approximation of the stress in an $H (\div)$-conforming space (see Chapter 9 of the
monograph \cite{BofBreFor:13} for an overview), this paper proposes an algorithm to obtain an $H (\div)$-conforming approximation of the
stress-tensor by post-processing the displacement-based approximation.

The idea of reconstructing the matrix-valued stress and vector-valued flux goes back to the hypercircle theorem by Prager and Synge
\cite{PraSyn:47} (see also Section III.9 in Braess' book \cite{Bra:07} for a presentation in modern mathematical language). Besides the accurate
approximation in an $H (\div)$-conforming space, the stress or flux reconstruction builds the basis of an a posteriori error estimator, which
was actually already one of the motivations of Prager and Synge \cite{PraSyn:47}.
Over the years, a posteriori error estimators based on flux reconstruction were explored in detail in many contributions
\cite{LadLeg:83,AinOde:93,LucWoh:04,CaiZha:12a,ErnVoh:15}. An important algorithmic innovation was given by Braess and Sch\"oberl
\cite{BraSch:08} by the equilibration procedure which is completely local and provides the link to residual error estimation. An important aspect
of the use of reconstruction-based error estimation of the above type is that it provides guaranteed upper bounds for the error with accessible
constants. Another important aspect is that these a posteriori error estimators are valid for any approximation that is inserted into the procedure.
In particular, it does not assume that the underlying finite-dimensional variational problems are solved to high precision.
The extension of reconstruction strategies to linear elasticity was the subject of a number of contributions in the last two decades
\cite{ParBonHuePer:06,NicWitWoh:08,Kim:11a,Kim:11b,AinAllBarRan:12}, stress reconstruction in the context of Stokes flow was also studied
recently \cite{HanSteVoh:12}. More recently, a posteriori error estimation based on the reconstruction of weakly symmetric stresses was investigated
in our earlier work \cite{BerMolSta:19a} and \cite{BerKobMolSta:19}. In particular, the stress equilibration procedure considered in our recent
contribution \cite{BerKobMolSta:19} serves as a point of departure for our treatment of hyperelastic material models in the present paper.
The recent paper by Botti and Riedlbeck \cite{BotRie:19} should also be mentioned here. It treats nonlinear elasticity restricted to a geometrically linear
situation. In that case, the (Piola-Kirchhoff) stress is still symmetric which allows the use of symmetric stress elements as it is done in the approach by
Botti and Riedlbeck \cite{BotRie:19}.

We emphasize once more that our paper does not discuss the issue of a posteriori error estimation. The development of an a posteriori error estimator
based on the stress equilibration for hyperelastic material models and, in particular, its analysis are expected to be rather involved and to require rather
restrictive assumptions. After all, it is well-known that the solution of the variational problem may not be unique (see the examples in Chapter 5 in \cite{Cia:88}).
We nevertheless hope that our stress reconstruction procedure will be of use for the future study of such an a posteriori error estimator.
For the time being, we concentrate on other motivations for the use of equilibrated stresses like the enhanced accuracy of surface force approximations
which will be studied in detail. Other approaches to the direct finite element approximation of stresses in geometrically nonlinear elasticity can be found
e.g. in \cite{HauHec:13} and \cite{MueStaSchSch:14}.

Besides the fact that the symmetry condition for the stresses becomes more complicated in the geometrically and materially nonlinear situation
associated with hyperelastic models which was already mentioned above, other challenging issues arise if one wants to extend the stress
equilibration procedure from our recent work \cite{BerKobMolSta:19} to that case. The stresses computed directly from displacement and, possibly,
pressure approximations are no longer piecewise polynomial due to the nonlinearity of the model. Therefore, in order to get a stress reconstruction
in an appropriate $H (\div)$-conforming finite element space, a suitable projection to piecewise polynomial stresses need to be carried out first.
Another problem is concerned with the subspace of test functions which are perpendicular to the range of the local equilibration systems for vertex
patches not connected to the Dirichlet boundary. The main result of this contribution is the identification of these subspaces as associated with rigid
body modes in the current configuration, i.e., involving the displacement approximations. The right-hand sides arising from straightforward piecewise
polynomial projections of the stresses are shown to have components outside of these ranges which means that the local equilibration systems
possess an additional compatibility error. We propose a remedy involving a more complicated test space to overcome this problem. This leads
to compatible local problems and thus to a truly equilibrated stress reconstruction.

The outline of this paper is as follows. We start with the variational formulation of elastic deformations governed by hyperelastic material models and
the weakly symmetric stress reconstruction in Section \ref{sec-hyperelasticity_stress}. Section \ref{sec-local_equilibration} presents the local
equilibration algorithm. The solvability of the local problems on vertex patches is analysed in \ref{sec-solvability}. In particular, the subspace of
test functions orthogonal to the range of the local operators associated with equilibration is identified and this result is used for the investigation
of the compatibility of the right-hand side.
Section \ref{sec-remedy} proposes our remedy to deal with this problem and derives a more complicated test space which leads to compatible
local equilibration systems for which an inf-sup condition holds.
The improved accuracy of the surface forces associated with the equilibrated stresses will be the topic of Section \ref{sec-surface_forces}.
Finally, computational results illustrating the properties of the equilibrated stresses are collected in Section \ref{sec-computational}.

\section{Hyperelasticity and weakly symmetric stress reconstruction}

\label{sec-hyperelasticity_stress}

The hyperelastic problems under our consideration are based on an open, bounded and connected domain
$\Omega \subset \R^d$ ($d = 2,3$) with Lipschitz-continuous boundary which constitutes the reference configuration of
the undeformed state. The boundary is divided into two disjoint non-empty subsets $\Gamma_D$ and $\Gamma_N$. On $\Gamma_D$,
homogeneous displacement boundary conditions $\bu = \bzero$ are imposed, while surface traction forces $\bP \cdot \bn = \bg$
are prescribed on $\Gamma_N$. For an appropriate subspace $\bV \subset H_{\Gamma_D}^1 (\Omega)^d$,
the boundary value problem of hyperelasticity then consists in the variational problem of finding
$\bu \in \bV$ such that
\begin{equation}
  ( \bP (\bu) , \bnabla \bv ) = ( \bff , \bv ) + \langle \bg , \bv \rangle_{0,\Gamma_N}
  \label{eq:variational_condition}
\end{equation}
holds for all $\bv \in \bV$. Here, $\bP(\bu) = \partial_\bF \psi (\bB)$ denotes the first Piola-Kirchhoff stress tensor with
respect to the stored energy function $\psi : \R^{d \times d}_{\text{sym}} \to \R$, where the deformation gradient is given by
$\bF (\bu) = \bI + \bnabla \bu$ and the left Cauchy-Green
strain tensor is defined as $\bB (\bu) = \bF (\bu) \bF(\bu)^T$. Simple brackets $( \: \cdot \: , \: \cdot \: )$ as in (\ref{eq:variational_condition}) will
from now on always abbreviate the inner product in $L^2 (\Omega)$ with respect to the reference configuration; $\bff$ and $\bg$ stand for volume
and surface loads, transformed back to the reference configuration. An example of a stored energy function which we will also use later in our
computations in Section \ref{sec-computational} is associated with the Neo-Hookean model
\begin{equation}
  \psi_{NH} (\bB) = \frac{1}{2} \left( \mu \: \tr \: \bB + \frac{\lambda}{2} \det (\bB) - \left( \mu + \frac{\lambda}{2} \right) \ln ( \det (\bB) ) \right) \: .
\end{equation}
In this case, the Piola-Kirchhoff stress tensor is given by
\begin{equation}
  \bP (\bu) = \partial_\bF \psi_{NH} (\bB (\bu)) = \mu \bF (\bu) + \left( \frac{\lambda}{2} (\det (\bB (\bu)) - 1) - \mu \right) \bF (\bu)^{-T}
  \label{eq:Neo_Hooke_stress_strain}
\end{equation}
and $\bV = W_{\Gamma_D}^{1,4} (\Omega)^d$ would be sufficient for the variational problem (\ref{eq:variational_condition}) to
be properly defined. In order to deal with materials in the incompressible parameter regime ($\lambda \gg \mu$), a pressure-like
variable may be introduced, e.g. by setting $p = \lambda (\det (\bF (\bu)) - 1)$. Note that with this choice, $p$ does not really stand for the
physical pressure but that it is possible to obtain the pressure from $p$ even in the incompressible limit. The above choice is motivated from
the fact that it turns into the constraint $p = \div \: \bu$, familiar from linear elasticity, in the small strain limit. Other options for the definition of
$p$ are possible and may have advantages. The Piola-Kirchhoff stress is now given in terms of $\bu$ and $p$ which, in the Neo-Hookean
example, reads
\begin{equation}
  \bP (\bu,p) = \mu \bF (\bu) + \left( p \left( 1 + \frac{p}{2 \lambda} \right) - \mu \right) \bF (\bu)^{-T}
  \label{eq:Neo_Hooke_pressure}
\end{equation}
due to the fact that $\det (\bB (\bu)) - 1 = \det (\bF (\bu))^2 - 1 = (\det (\bF (\bu)) - 1) (\det (\bF (\bu)) + 1)$ holds. With a pressure space $Q$
($Q = L^{4/3} (\Omega)$ would be appropriate in the Neo-Hooke case), the variational problem turns into
one of saddle point type which consists in finding $\bu \in \bV$ and $p \in Q$ such that
\begin{equation}
  \begin{split}
    ( \bP (\bu,p) , \bnabla \bv ) & = ( \bff , \bv ) + \langle \bg , \bv \rangle_{0,\Gamma_N} \mbox{ for all } \bv \in \bV \: ,\\
    ( \det (\bF (\bu)) - 1 , q ) - \frac{1}{\lambda} ( p , q ) & = 0 \hspace{2.3cm} \mbox{ for all } q \in Q^\prime
  \end{split}
  \label{eq:variational_condition_pressure}
\end{equation}
with $Q^\prime$ denoting the dual space of $Q$ ($Q^\prime = L^4 (\Omega)$ with the above choices for the Neo-Hooke case).

For $k \geq 1$, let $\bV_h \subset \bV$ be the subspace of continuous piecewise polynomials of degree $k+1$ with
respect to a triangulation $\cT_h$ for each component of $\bV_h$. Our finite-dimensional variational problem for hyperelasticity consists in finding
$\bu_h \in \bV_h$ such that
\begin{equation}
  ( \bP (\bu_h) , \nabla \bv_h ) = ( \bff , \bv_h ) + \langle \bg , \bv_h \rangle_{0,\Gamma_N}
  \label{eq:Galerkin_condition}
\end{equation}
holds for all $\bv_h \in \bV_h$. In the incompressible regime, a discrete pressure space $Q_h$ consisting of continuous piecewise polynomials of
degree $k$ may be used to define a corresponding discrete saddle point problem. It consists in finding $(\bu_h , p_h) \in \bV_h \times Q_h$ such
that
\begin{equation}
  \begin{split}
    ( \bP (\bu_h,p_h) , \bnabla \bv_h ) & = ( \bff , \bv_h ) + \langle \bg , \bv_h \rangle_{0,\Gamma_N} \mbox{ for all } \bv_h \in \bV_h \: ,\\
    ( \det (\bF (\bu_h)) - 1 , q_h ) - \frac{1}{\lambda} ( p_h , q_h ) & = 0 \hspace{2.25cm} \mbox{ for all } q_h \in Q_h
  \end{split}
  \label{eq:Galerkin_condition_pressure}
\end{equation}
is satisfied. The direct use of $\bP (\bu_h)$ or, in the incompressible regime, $\bP (\bu_h,p_h)$ as an approximation for the Piola-Kirchhoff stress,
has, however, certain deficiencies which are already known from the linear elasticity situation. Most importantly, $\bP (\bu_h) \cdot \bn$ is not
continuous at interfaces between elements of the underlying triangulation implying that traction forces are not well-defined. It also means that
$\bP (\bu_h)$ is not $H (\div)$-conforming and that the conservation of momentum is not controlled. This motivates the need to construct an
$H (\div)$-conforming stress reconstruction $\bP_h^R$ with all these desired properties.

The idea of equilbration is to compute the reconstructed stress $\bP_h^R$ in the $H (\div)$-conforming Raviart-Thomas space of degree $k$ as an
additive correction to $\bP (\bu_h)$. This is done using the broken Raviart-Thomas space of degree $k$ for each row leading to
\begin{align*}
  \bPi_h^\Delta
  = \{ \bP_h : \Omega \rightarrow \R^{d \times d} \mbox{ with } \left. \bP_h \right|_T \in P_k (T)^{d \times d} + P_k (T)^d \bx^T \} \: ,
\end{align*} 
where $P_k (T)$ denotes the space of polynomials of degree $k$ on the triangle ($d=2$) or tetrahedron ($d=3$) $T$. In other words, each row of the
stress tensor $\bP_h \in \bPi_h^\Delta$ is element-wise given by a function in the Raviart-Thomas space. Unfortunately, in contrast to the linear elasticity
situation, $\bP (\bu_h) \in \bPi^\Delta$ does not hold, in general, due to the nonlinearity of the stress-strain relation. Obviously, for the Neo-Hookean
model in (\ref{eq:Neo_Hooke_stress_strain}), $\bP (\bu_h)$ is not even piecewise polynomial. Therefore, $\bP (\bu_h)$ needs to be projected first
to an element $\widehat{\bP}_h (\bu_h) \in \bPi_h^\Delta$. An obvious candidate would be to set $\widehat{\bP}_h (\bu_h) = \cP_h^k \bP (\bu_h)$,
where $\cP_h^k$ denotes the component-wise and element-wise $L^2$-orthogonal projection onto $P_k (T)$. We will stick with this choice of
$\widehat{\bP}_h (\bu_h)$ for the moment until we present an alternative one in Section \ref{sec-remedy} as a remedy for certain deficiencies
associated with it.

Following the weakly symmetric equilibration procedure from \cite{BerKobMolSta:19}, we perform the construction for the difference
${\bP}_h^\Delta := \bP_h^R - \widehat{\bP}_h (\bu_h)$ between the reconstructed and the projected original stress.
Recall that the extension of the hypercircle theorem to linear elasticity requires a symmetric reconstruction satisfying the
equilibration condition $\div \: \bP_h^\Delta= -  \bff - \div \: \widehat\bP_h (\bu_h)$ in each triangle and the jump condition
allowing $\bP_h^R$ to be $H (\div)$-conforming. In order to write this jump condition in a precise way, let $\cS_h$ denote the set of all sides
(edges in 2D and faces in 3D) of the triangulation $\cT_h$ and $\cS_h^\ast$ the set of sides not contained in $\Gamma_D$
\[
  \cS_h^\ast := \{ S \in \cS_h: S \nsubseteq \Gamma_D \} \: .
\]
Further, for all sides $S \in \cS_h$, let $\bn$ be the normal direction associated with $S$ (depending on its orientation), $T_+$ and $T_-$  the
elements adjacent to $S$ (such that $\bn$ points into $T_+$) and the jump of $\bP_h$ over $S$ defined by
\begin{equation}
  \llbracket \bP_h \cdot \bn \rrbracket_S = \left. \bP_h \cdot \bn \right|_{T_-} - \left. \bP_h \cdot \bn \right|_{T_+} \: .
  \label{eq:definition_jump}
\end{equation}
For sides $S \subset \Gamma_N$ located on the Neumann boundary we assume that $\bn$ points outside of $\Omega$ and define the jump by
\[
  \llbracket \bP_h \cdot \bn \rrbracket_S = \left. \bP_h \cdot \bn \right|_{T_-} \: .
\]
In order to use the same formulas also for patches adjacent to the Neumann boundary $\Gamma_N$ we define the auxiliary jump by
\begin{equation}
  \llbracket \bP_h \cdot \bn \rrbracket_S^\ast = \left\{ \begin{array}{lcr}
    \left. \bP_h \cdot \bn \right|_{T_-} - \bg & , \mbox{ if } & S \subset \Gamma_N \: , \\
    \llbracket \bP_h \cdot \bn \rrbracket_S & , \mbox{ if } & S \nsubseteq \Gamma_N \: .
  \end{array} \right.
  \label{eq:definition_jump_N}
\end{equation}
With this, the jump condition for the correction reads
$\llbracket \bP_h^\Delta \cdot \bn \rrbracket_S = -  \llbracket  \widehat\bP_h (\bu_h) \cdot \bn \rrbracket_S^\ast$ for all sides $S \in \cS_{h}^\ast$.

Similarly as in \cite{BerKobMolSta:19}, the symmetry condition will be imposed weakly in order to obtain a reconstructed stress with reasonable
symmetry properties. In the hyperelastic setting, symmetry does not hold for $\bP (\bu)$ but instead for the related Cauchy stress tensor
$\bsigma (\bu) = \bP (\bu) \bF (\bu)^T / \det (\bF (\bu))$ which adequately describes stresses in the deformed configuration. Rewritting the
equilibration and jump conditions in a weak form and applying the weak symmetry condition to $\bP_h^R \bF (\bu)^T$ leads to the following
conditions for $\bP_h^\Delta$:
\begin{equation}
  \begin{split}
    ( \div \: \bP_h^\Delta , \bz_h )_{h}
    & = - ( \bff + \div \: \widehat\bP_h (\bu_h) , \bz_h )_{h} \hspace{0.4cm} \mbox{ for all } \bz_h \in \bZ_h \: , \\
    \langle \llbracket \bP_h^\Delta \cdot \bn \rrbracket_S , \bzeta \rangle_S
    & = - \langle \llbracket  \widehat\bP_h (\bu_h) \cdot \bn \rrbracket_S^\ast , \bzeta \rangle_S \hspace{0.6cm} \mbox{ for all }
    \bzeta \in P_k (S)^d \: , \: S \in \cS_h^\ast \: , \\
    ( \bP_h^\Delta \bF(\bu_h)^T , \bJ (\bgamma_h) ) & = - ( \widehat{\bP}_h (\bu_h) \bF(\bu_h)^T , \bJ (\bgamma_h) )
    \mbox{ for all } \bgamma_h \in \bX_h \: .
  \end{split}
  \label{eq:equilibration_conditions}
\end{equation}
$\bZ_h$ may be chosen to be the space of discontinuous $d$-dimensional vector functions which are piecewise polynomial of degree $k$, and
$\bX_h$ may stand for the continuous $d (d-1)/2$-dimensional vector functions which are piecewise polynomial of degree $k$ with $\bJ (\btheta)$
being defined by 
\begin{equation}
  \bJ (\theta) := \begin{pmatrix} \;\;0 & \theta \\ -\theta & 0 \end{pmatrix} \mbox{ for } d = 2 \mbox{ and }
  \bJ (\btheta) :=
  \begin{pmatrix} \;\;0 & \;\;\theta_3 & -\theta_2 \\ -\theta_3 & \;\;0 & \;\;\theta_1 \\ \;\;\theta_2 & -\theta_1 & \;\;0 \end{pmatrix}
  \mbox{ for } d = 3
\label{eq:skew_symmetric_tensor}
\end{equation}
for every $d (d-1)/2$-dimensional vector $\btheta$. This choice is motivated by the inf-sup stability of the corresponding combination with the use
of Raviart-Thomas element of degree $k \geq 1$ as stress approximation space in the Hellinger-Reissner formulation (see Boffi, Brezzi and Fortin
\cite{BofBreFor:09}). 

\section{Local stress equilibration algorithm}

\label{sec-local_equilibration}

For the sake of the efficient computation of the stress reconstruction, we localize the problem using a partition of unity. The commonly
used partition of unity with respect to the set $\cV_h$ of all vertices of $\cT_h$,
\begin{equation}
  1 \equiv \sum_{z \in \cV_h} \tilde{\phi}_z \mbox{ on } \Omega \: ,
  \label{eq:partition_of_unity_1}
\end{equation}
consists of continuous piecewise linear functions $\tilde{\phi}_z$. In this case, the support of $\tilde{\phi}_z$ is restricted to
\begin{equation}
  \tilde{\omega}_z := \bigcup \{ T \in \cT_h : z \mbox{ is a vertex of } T \} \: .
  \label{eq:vertex_patch}
\end{equation}
In analogy to the stress equilibration procedure described in \cite{BerKobMolSta:19} for the linear elasticity case, we modify this classical partition
of unity in order to exclude patches formed by vertices $z \in \Gamma_N$, where the local problems may possess to few degrees of freedom to
be solvable. To this end, let $\cV_h^\prime = \{ z \in \cV_h : z \notin \Gamma_N \}$ denote the subset of vertices which are not located on a side
(edge/face) of $\Gamma_N$. The modified partition of unity is defined by
\begin{equation}
  1 \equiv \sum_{z \in \cV_h^\prime} \phi_z \mbox{ on } \Omega \: .
  \label{eq:partition_of_unity}
\end{equation}
For $z \in \cV_h^\prime$ not connected by an edge to $\Gamma_N$ the function $\phi_z$ is equal to $\tilde{\phi}_z$. Otherwise,
the function $\phi_z$ has to be modified in order to account for unity at the connected vertices on $\Gamma_N$. For each
$z_N \in \Gamma_N$ one vertex $z_I \notin \Gamma_N$ connected by an edge with $z_N$ is chosen and $\tilde{\phi}_{z_I}$ is
extended by the value $1$ along the edge from $z_I$ to $z_N$ to obtain the modified function $\phi_{z_I}$. The support of
$\phi_z$ is denoted by
\begin{equation}
  \omega_z := \bigcup \{ T \in \cT_h : \phi_z = 1 \mbox{ for at least one vertex } z \mbox{ of } T \} \: .
  \label{eq:vertex_patch_prime}
\end{equation}
For the partition of unity (\ref{eq:partition_of_unity}) to hold, we require the triangulation $\cT_h$ to be such that each vertex on
$\Gamma_N$ is connected to an interior edge. For the localized equilibration algorithm, we will also need the local subspaces
\begin{equation}
  \bPi_{h,z}^{\Delta} = \{ \bq_h \in \bPi_h^{\Delta} : \bq_h \cdot \bn = \bzero \mbox{ on } \partial \omega_z \backslash \partial \Omega \: , \:
  \bq_h \equiv \bzero \mbox{ on } \Omega \backslash \overline{\omega}_z \}
 \label{eq:local_subspace}
\end{equation}
for all $z \in \cV_h^\prime$. Moreover, we need to work with the local sets of sides
$\cS_{h,z} := \{ S \in \cS_h : S \subset \overline{\omega}_z \}$ and the restrictions $\bZ_{h,z}$ and $\bX_{h,z}$ to $\omega_z$ of the test spaces
$\bZ_h$ and $\bX_h$, respectively. The conditions in (\ref{eq:equilibration_conditions}) can be restated for a sum of patch-wise contributions
\begin{equation}
  \bP_h^\Delta = \sum_{z \in \cV_h^\prime} \bP_{h,z}^\Delta \: ,
  \label{eq:patch_decomposition}
\end{equation}
where, for each $z \in \cV_h^\prime$, $\bP_{h,z}^\Delta \in  \bPi_{h,z}^{\Delta}$ is computed such that
$\Vert \bP_{h,z}^\Delta \Vert_{\omega_z}^2 $ is minimized subject to the following constraints:
\begin{equation}
  \begin{split}
    ( \div \: \bP_{h,z}^\Delta , \bz_{h,z} )_{\omega_z,h}
    & = - ( ( \bff + \div \: \widehat{\bP}_h (\bu_h) ) \phi_z , \bz_{h,z}  )_{\omega_z,h}
    \mbox{ for all } \bz_{h,z} \in \bZ_{h,z} \: , \\
    \langle \llbracket \bP_{h,z}^\Delta \cdot \bn \rrbracket_S , \bzeta \rangle_S
    & = - \langle \llbracket \widehat{\bP}_h (\bu_h) \cdot \bn \rrbracket_S \, \phi_z , \bzeta \rangle_S
    \hspace{0.85cm} \mbox{ for all } \bzeta \in P_k (S)^d \: , \: S \in \cS_{h,z} \: , \\
    ( \bP_{h,z}^\Delta \bF (\bu_h)^T , \bJ (\bgamma_{h,z}) )_{\omega_z} &
    = - ( \widehat{\bP}_h (\bu_h) \bF (\bu_h)^T \phi_z , \bJ (\bgamma_{h,z}) )_{\omega_z}
    \hspace{0.4cm} \mbox{ for all } \bgamma_{h,z} \in \bX_{h,z} \: .
  \end{split}
  \label{eq:equilibration_conditions_local}
\end{equation}
The minimization in (\ref{eq:equilibration_conditions_local}) is necessary since solutions to (\ref{eq:equilibration_conditions_local}) are not expected
to be unique, in general, similarly to the linear elasticity case treated in our earlier work \cite{BerKobMolSta:19}.
At this point, we may introduce the local orthogonal projections $\cP_{h,z}^k : L^2 (\omega_z) \rightarrow \bZ_{h,z}$ and
$\cP_{h,S}^k : L^2 (S) \rightarrow P_k (S)^d$ which means that the first two conditions in (\ref{eq:equilibration_conditions_local}) can be written
shortly as
\begin{align*}
  \div \: \bP_{h,z}^\Delta & = - \cP_{h,z}^k ((\bff + \div \: \widehat{\bP} (\bu_h)) \phi_z) \: , \\
  \llbracket \bP_{h,z}^\Delta \cdot \bn \rrbracket_S & = - \cP_{h,S}^k (\llbracket \widehat{\bP}_h (\bu_h) \cdot \bn \rrbracket_S \phi_z) \: .
\end{align*}
For each $z \in \cV_h^\prime$, (\ref{eq:equilibration_conditions_local}) constitutes a low-dimensional quadratic minimization problem with linear
constraints for which standard methods are available for the efficient solution. Note that it is not guaranteed at this point that
(\ref{eq:equilibration_conditions_local}) has a solution at all. In fact, it does not, in general, as will become clear from the results of the next section.
This is the reason why we will modify the test space in Section \ref{sec-remedy} in order to have well-posed local patch problems.

To get an idea about the structure of the system (\ref{eq:equilibration_conditions_local}) and as a motivation for the result in the next section, we
consider its underlying continuous problem. On the continuous level, the system (\ref{eq:equilibration_conditions_local}) constitutes the stress-based
dual formulation of the variational problem (\ref{eq:variational_condition}) restricted to $\omega_z$. With a suitable subspace
$\bV_z \subset H_{\Gamma_D \cap \partial \omega_z}^1 (\omega_z)^d$ this means that $\bz \in \bV_z$ is sought such that
\begin{equation}
  ( \bP (\bz) , \bnabla \bv )_{\omega_z} = ( \bff , \bv )_{\omega_z} + \langle \bg , \bv \rangle_{\partial \omega \cap \Gamma_N}
  \mbox{ for all } \bv \in \bV_z
  \label{eq:local_continuous_problem}
\end{equation}
holds. On vertex patches with $\Gamma_D \cap \partial \omega_z = \emptyset$, there is a non-trivial subspace
$\bV_z^\circ \subset \bV_z$ of test functions such that $( \bP (\bz) , \bnabla \bv ) = 0$ for $\bv \in \bV_z^\circ$. Obviously, all constants are
contained in $\bV_z^\circ$. Moreover, since
\[
  ( \bP (\bz) , \nabla \bv )_{\omega_z} = ( \bP (\bz) \bF (\bz)^T , \bnabla \bv \bF (\bz)^{-1} )_{\omega_z}
\]
holds and since $\bP (\bz) \bF (\bz)^T$ is a symmetric matrix, also all $\bv$ with $\bnabla \bv \bF (\bz)^{-1}$ being skew-symmetric will be contained
in $\bV_z^\circ$. In two dimensions, we arrive at
\begin{equation}
  \mbox{span} \{ \begin{pmatrix} 1 \\ 0 \end{pmatrix} , \begin{pmatrix} 0 \\ 1 \end{pmatrix} , \begin{pmatrix} x_2 + z_2 \\ - (x_1 + z_1) \end{pmatrix} \}
  \label{eq:2D_RM}
\end{equation}
being contained in $\bV^\circ$, and for $d = 3$ this is true for
\begin{equation}
  \mbox{span} \{ \begin{pmatrix} 1 \\ 0 \\ 0 \end{pmatrix} , \begin{pmatrix} 0 \\ 1 \\ 0 \end{pmatrix} , \begin{pmatrix} 0 \\ 0 \\ 1 \end{pmatrix} ,
  \begin{pmatrix} 0 \\ x_3 + z_3 \\ - (x_2 + z_2) \end{pmatrix} , \begin{pmatrix} - (x_3 + z_3) \\ 0 \\ x_1 + z_1 \end{pmatrix} ,
  \begin{pmatrix} x_2 + z_2 \\ - (x_1 + z_1) \\ 0 \end{pmatrix} \} \: .
  \label{eq:3D_RM}
\end{equation}
These are exactly the rigid body modes associated with the current configuration deformed by $\bvarphi (\bx) = \bx + \bz$ which we would like to
denote by $\bR\bM (\bz)$ from now on. From the above derivation, it should not be surprising that the corresponding rigid body mode spaces
$\bR\bM (\bu_h)$ will appear in the investigation of the well-posedness of the discrete local problems (\ref{eq:equilibration_conditions_local}) in
the following section.

\section{Solvability of the local problems on vertex patches}

\label{sec-solvability}

We turn our attention to the solvability of the local minimization problem subject to the constraints (\ref{eq:equilibration_conditions_local}). To this
end, we need to guarantee that for every right hand side, a function $\bP_{h,z}^\Delta \in \bPi_{h,z}^\Delta$ exists such that the
constraints (\ref{eq:equilibration_conditions_local}) are satisfied. The left-hand side in (\ref{eq:equilibration_conditions_local}) defines a linear
operator $\cL_{h,z} : \Pi_{h,z}^\Delta \rightarrow \bZ_{h,z}^\prime \times \bS_{h,z}^\prime \times \bX_{h,z}^\prime$, where
$\bS_{h,z} = \{ \bzeta \in P_k (S)^d : S \in \cS_{h,z} \}$ denotes the trace space on the interior sides and $( \: \cdot \: )^\prime$ stands for the dual
space. The subspace $\bR_{h,z}^\perp \subseteq \bZ_{h,z} \times \bS_{h,z} \times \bX_{h,z}$ orthogonal to the range of $\cL_{h,z}$, i.e., the
null space of its adjoint $\cL_{h,z}^\ast$, is obviously of interest for the solvability since the linear functionals on the right-hand side in
(\ref{eq:equilibration_conditions_local}) need to vanish on $\bR_{h,z}^\perp$.
This subspace can be characterized as follows.

\begin{proposition}
  The subspace
  \begin{equation}
    \begin{split}
      \bR_{h,z}^\perp = \{ & ( \bz_{h,z} , \bs_{h,z} , \bgamma_{h,z} ) \in \bZ_{h,z} \times \bS_{h,z} \times \bX_{h,z} : \\
      & ( \div \: \bP_{h,z}^\Delta , \bz_{h,z} )_{\omega_z,h}
      - \sum_{S \in \cS_{z,h}} \langle \llbracket \bP_{h,z}^\Delta \cdot \bn \rrbracket_S , \bs_{h,z} \rangle_S
      + ( \bP_{h,z}^\Delta , \bJ (\bgamma_{h,z}) \bF (\bu_h) )_{\omega_z} = 0
      \mbox{ for all } \bP_{h,z}^\Delta \in \bPi_{h,z}^\Delta \} \: ,
    \end{split}
    \label{eq:adjoint_null_space}
  \end{equation}
  i.e., the null space of the adjoint operator $\cL_{h,z}^\ast$ associated with the constraints (\ref{eq:equilibration_conditions_local}), can be
  characterized as follows:
  \begin{equation}
    \begin{split}
      \bR_{h,z}^\perp & = \{ ( \cP_{h,z}^k \brho , \{ \cP_{h,S}^k \brho \}_{S \in \cS_{h,z}} , \btheta ) :
      ( \brho , \btheta ) \in \bR\bM (\bu_h) \times \R^{d (d-1)/2} \mbox{ such that } \bJ (\btheta) \bF (\bu_h) = \bnabla \brho \}
      \mbox{ if } | \partial \omega_z \cap \Gamma_D | = 0 \: , \\
      \bR_{h,z}^\perp & = \{ ( \bzero , \bzero , \bzero ) \} \mbox{ if } | \partial \omega_z \cap \Gamma_D | > 0 \: .
    \end{split}
    \label{eq:adjoint_null_space_characterization}
  \end{equation}
  Here, $| \: \cdot \: |$ denotes the $d-1$-dimensional measure of boundary curves or surfaces, respectively.
  \label{prop-adjoint_null_space_characterization}
\end{proposition}

\begin{proof}
  The proof is carried out for $d = 3$; the two-dimensional case is much easier and can be derived from the three-dimensional one in the usual way
  by setting $u_3 \equiv 0$ and all other functions to be independent of $x_3$ (with appropriate modifications of operators such as
  $\div$, $\bnabla$, $\bcurl$, etc.).
  
  {\em 1st Step.} We start by showing that the component $\bgamma_{h,z}$ of $\bR_{h,z}^\perp$ in (\ref{eq:adjoint_null_space}) needs to satisfy
  \begin{equation}
    \bJ (\bgamma_{h,z}) \bF (\bu_h) = \nabla (\gamma_1 \brho_1 + \gamma_2 \brho_2 + \gamma_3 \brho_3) \: .
    \label{eq:rotation_and_gradient}
  \end{equation}
  Let us restrict ourselves to the $H (\div)$-conforming subspace of $\bPi_{h,z}^\Delta$, i.e., with the property that
  $\llbracket \bP_{h,z}^\Delta \cdot \bn \rrbracket_S = \bzero$ for all $S \in \cS_{h,z}$.
  Then, the condition in (\ref{eq:adjoint_null_space}) for the definition of $\bR_{h,z}^\perp$ turns into
  \begin{equation}
    ( \div \: \bP_{h,z}^\Delta , \bz_{h,z} )_{\omega_z}
    + ( \bP_{h,z}^\Delta , \bJ (\bgamma_{h,z}) \bF (\bu_h) )_{\omega_z} = 0 \: .
    \label{eq:condition_conforming}
  \end{equation}
  By definition, we can write
  \begin{equation}
    \begin{split}
      \bJ (\bgamma_{h,z}) & \bF (\bu_h) =
      \begin{pmatrix} 0 & \gamma_3 & -\gamma_2 \\ -\gamma_3 & 0 & \gamma_1 \\ \gamma_2 & -\gamma_1 & 0 \end{pmatrix}
      \bnabla \begin{pmatrix} x_1 + u_1 \\ x_2 + u_2 \\ x_3 + u_3 \end{pmatrix} \\
      & = \gamma_1 \bnabla \begin{pmatrix} 0 \\ x_3 + u_3 \\ -(x_2 + u_2) \end{pmatrix}
      + \gamma_2 \bnabla \begin{pmatrix} -(x_3 + u_3) \\ 0 \\ x_1 + u_1 \end{pmatrix}
      + \gamma_3 \bnabla \begin{pmatrix} x_2 + u_2 \\ -(x_1 + u_1) \\ 0 \end{pmatrix} \\
      & = \gamma_1 \bnabla \brho_1 + \gamma_2 \bnabla \brho_2 + \gamma_3 \bnabla \brho_3 \: .
    \end{split}
    \label{eq:gradient_rotations}
  \end{equation}
  We may restrict ourselves further to divergence-free $\bP_{h,z}^\Delta$ with $\bP_{h,z} \cdot \bn = \bzero$ on the entire boundary
  $\partial \omega_z$. These stress approximations can be written as
  $\bP_{h,z}^\Delta = \bcurl \: \bpsi_{h,z}$ with $\bpsi_{h,z}$ in the N\'{e}d\'{e}lec space $N_k (\cT_h)^d$
  (cf. \cite[Corollary 2.3.2]{BofBreFor:13}) with boundary conditions $\bn \times \bpsi_{h,z} = \bzero$ on $\partial \omega_z$.
  Inserting this into (\ref{eq:condition_conforming}) and integrating by parts leads to
  \begin{equation}
    \begin{split}
      0 & = ( \bP_{h,z}^\Delta , \bJ (\bgamma_{h,z}) \bF (\bu_h) )_{\omega_z}
            = ( \bP_{h,z}^\Delta , \gamma_1 \bnabla \brho_1 + \gamma_2 \bnabla \brho_2 + \gamma_3 \bnabla \brho_3 )_{\omega_z} \\
         &
         = ( \bcurl \: \bpsi_{h,z} , \gamma_1 \bnabla \brho_1 + \gamma_2 \bnabla \brho_2 + \gamma_3 \bnabla \brho_3 )_{\omega_z} \\
         &=( \bpsi_{h,z} , \bcurl (\gamma_1 \bnabla \brho_1+\gamma_2 \bnabla \brho_2 + \gamma_3 \bnabla \brho_3) )_{\omega_z} \\
         & = ( \bpsi_{h,z} , \nabla \gamma_1 \times \bnabla \brho_1 + \nabla \gamma_2 \times \bnabla \brho_2
         + \nabla \gamma_3 \times \bnabla \brho_3 )_{\omega_z} \: ,
    \end{split}
    \label{eq:condition_conforming_divergence_free}
  \end{equation}
  where we used the fact that $\bcurl \: \bnabla \brho_1 = \bcurl \: \bnabla \brho_2 = \bcurl \: \bnabla \brho_3 = \bzero$. It can be
  shown that (\ref{eq:condition_conforming_divergence_free}) can only hold for all $\bpsi_{h,z}$ if 
  $\nabla \gamma_1 = \nabla \gamma_2 = \nabla \gamma_3 = \bzero$ in the following way:
  In the lowest-order case $k = 1$, one may insert as test functions $\bpsi_{h,z}$ with tangential component $\bpsi_{h,z} \cdot \bt_E \equiv \be_i$
  for $i = 1 , 2 , 3$ on an interior edge $E \subset \omega_z \backslash \partial \omega_z$ and $\bpsi_{h,z} \cdot \bt_{E^\prime} \equiv \bzero$ on
  all the other interior edges $E^\prime$. If $(\# E)_z$ denotes the number of interior edges in $\omega_z$, this gives $3 (\# E)_z$ linearly
  independent conditions for the $3 (\# E)_z$ constant values $(\nabla \gamma_i) \cdot \bt_E$ for $i = 1 , 2 , 3$ on all interior edges $E$.
  Therefore, (\ref{eq:condition_conforming_divergence_free}) implies that the tangential derivatives of $\gamma_1$, $\gamma_2$ and $\gamma_3$
  vanish along all interior edges $E$ which implies that $\gamma_1$, $\gamma_2$ and $\gamma_3$ are themselves constant.
  For the higher-order case, each increase of the polynomial degree from $k-1$ to $k$ gives additional degrees of freedom to be controlled:
  For each of the three components, one per edge (including edges on $\partial \omega_z$), additionally $k-2$ per face (including faces on
  $\partial \omega_z$) and additionally $(k-2) (k-3) / 2$ per tetraeder. This is more than compensated for by the additional test functions available
  in the N\'{e}d\'{e}lec space $N_k (\cT_h)$, see \cite[Proposition 2.3.5]{BofBreFor:13} so that $\gamma_1$, $\gamma_2$ and $\gamma_3$ are
  still forced by (\ref{eq:condition_conforming_divergence_free}) to remain constant. Finally, the fact that $\gamma_1$, $\gamma_2$ and
  $\gamma_3$ need to be constant implies that (\ref{eq:gradient_rotations}) can be written as (\ref{eq:rotation_and_gradient}).

  {\em 2nd Step.} Inserting (\ref{eq:rotation_and_gradient}) into (\ref{eq:condition_conforming}) and, restricting ourselves to
  $\bP_{h,z} \in \bPi_{h,z}^\Delta$ with, in addition to $\llbracket \bP_{h,z} \cdot \bn \rrbracket_S = \bzero$ for all $S \in \cS_{h,z}$,
  $\bP_{h,z} \cdot \bn = \bzero$ on all of $\partial \omega_z$ (which is automatically satisfied if $\partial \omega_z \cap \Gamma_D = \emptyset$),
  integration by parts leads to
  \begin{equation}
    ( \div \: \bP_{h,z}^\Delta , \bz_{h,z} )_{\omega_z}
    = ( \div \: \bP_{h,z}^\Delta , \gamma_1 \brho_1 + \gamma_2 \brho_2 + \gamma_3 \brho_3 + \ba)_{\omega_z}
    \label{eq:condition_conforming_div}
  \end{equation}
  with an arbitrary constant $\ba \in \R^3$. The range of the divergence operator satisfies
  \begin{align*}
    \{ \div \: \bP_{h,z}^\Delta : \bP_{h,z} \in \bPi_{h,z}^\Delta \mbox{ with } \llbracket \bP_{h,z}^\Delta \cdot \bn \rrbracket_S = 0 &
    \mbox{ for all } S \in \cS_h \mbox{ and } \bP_{h,z} \cdot \bn = \bzero \mbox{ on } \partial \omega_z \} \\
    & = \{ \bz_{h,z} \in \bZ_{h,z} : ( \bz_{h,z} , \be_i )_{\omega_z} = 0 \mbox{ for } i = 1 , \ldots , d \} =: \bZ_{h,z}^0 \: .
  \end{align*}
  If $\cP_{h,z}^{k,0}$ denotes the $L^2 (\omega_z)$-orthogonal projection to $\bZ_{h,z}^0$, then (\ref{eq:condition_conforming_div}) implies that
  $\bz_{h,z} = \cP_{h,z}^{k,0} (\gamma_1 \brho_1 + \gamma_2 \brho_2 + \gamma_3 \brho_3) + \ba$ which means that
  $\bz_{h,z} = \cP_{h,z}^k (\gamma_1 \brho_1 + \gamma_2 \brho_2 + \gamma_3 \brho_3 + \widetilde{\ba})$ with some $\widetilde{\ba} \in \R^3$.
  Since all rigid body modes $\brho \in \bR\bM (\bu_h)$ which can be written as
  $\brho = \gamma_1 \brho_1 + \gamma_2 \brho_2 + \gamma_3 \brho_3 + \widetilde{\ba}$, we have the corresponding representation of
  $\bz_{h,z}$ in (\ref{eq:adjoint_null_space_characterization}).
  
  {\em 3rd Step.} Now we need to consider the two cases in (\ref{eq:adjoint_null_space_characterization}) separately. If
  $| \partial \omega_z \cap \Gamma_D | = 0$, we have indeed that every pair $( \brho , \btheta ) \in \bR\bM (\bu_h) \times \R^{d (d-1)/2}$ with
  $\bJ (\btheta) \bF (\bu_h) = \bnabla \brho$ gives rise to a solution of (\ref{eq:adjoint_null_space}) in the form
  $( \cP_{h,z}^k \brho , \{ \cP_{h,S}^k \brho \}_{S \in \cS_{h,z}} , \btheta )$. This is due to the fact that, for all
  $\bP_{h,z}^\Delta \in \bPi_{h,z}^\Delta$,
  \begin{align*}
    ( \div \: \bP_{h,z}^\Delta , \cP_{h,z}^k \brho )_{\omega_z,h}
    & - \sum_{S \in \cS_{z,h}} \langle \llbracket \bP_{h,z}^\Delta \cdot \bn \rrbracket_S , \cP_{h,S}^k \brho \rangle_S
    + ( \bP_{h,z}^\Delta , \bJ (\btheta) \bF (\bu_h) )_{\omega_z} \\
    & = ( \div \: \bP_{h,z}^\Delta , \brho )_{\omega_z,h}
    - \sum_{S \in \cS_{z,h}} \langle \llbracket \bP_{h,z}^\Delta \cdot \bn \rrbracket_S , \brho \rangle_S
    + ( \bP_{h,z}^\Delta , \bJ (\btheta) \bF (\bu_h) )_{\omega_z} \\
    & = - ( \bP_{h,z}^\Delta , \bnabla \brho )_{\omega_z} + ( \bP_{h,z}^\Delta , \bJ (\btheta) \bF (\bu_h) )_{\omega_z} = 0
  \end{align*}
  holds. On the other hand, in the case $| \partial \omega_z \cap \Gamma_D | > 0$,
  \begin{align*}
    0 & = ( \bP_{h,z}^\Delta , - \bnabla \brho + \bJ (\btheta) \bF (\bu_h) )_{\omega_z} \\
    & = ( \div \: \bP_{h,z}^\Delta , \brho )_{\omega_z,h}
    - \sum_{S \in \cS_{z,h}} \langle \llbracket \bP_{h,z}^\Delta \cdot \bn \rrbracket_S , \brho \rangle_S
    - \sum_{S \subset \Gamma_D} \langle \bP_{h,z}^\Delta \cdot \bn , \brho \rangle_S
    + ( \bP_{h,z}^\Delta , \bJ (\btheta) \bF (\bu_h) )_{\omega_z} \\
    & = ( \div \: \bP_{h,z}^\Delta , \cP_{h,z}^k \brho )_{\omega_z,h}
    - \sum_{S \in \cS_{z,h}} \langle \llbracket \bP_{h,z}^\Delta \cdot \bn \rrbracket_S , \cP_{h,S}^k \brho \rangle_S
    - \sum_{S \subset \Gamma_D} \langle \bP_{h,z}^\Delta \cdot \bn , \cP_{h,S}^k \brho \rangle_S
    + ( \bP_{h,z}^\Delta , \bJ (\btheta) \bF (\bu_h) )_{\omega_z}
  \end{align*}
  holds for all $\bP_{h,z}^\Delta \in \bPi_{h,z}^\Delta$. Choosing $\bP_{h,z}^\Delta \in \bPi_{h,z}^\Delta$ appropriately, this implies that
  $\cP_{h,S}^k \brho = \bzero$ must hold on all $S \subset \Gamma_D$. Since there is at least one side $S \subset \Gamma_D$ and due to the
  special structure of the space $\bR\bM (\bu_h)$, the only possibility is $\brho = \bzero$.
  
\end{proof}

\begin{remark}
  In the linear elasticity case, Proposition 1 turns into the corresponding result from our earlier work \cite{BerKobMolSta:19}, where
  $( \bz_{h,z} , \bs_{h,z} , \bgamma_{h,z} ) = ( \brho , \{ \left. \brho \right|_S \}_{S \in \cS_h} , \btheta )$ for
  $( \brho , \btheta ) \in \bR\bM (\bu_h) \times \R^{d (d-1)/2}$ with $\bJ (\theta) = \bnabla \brho$.
\end{remark}

Basic linear algebra tells us that the right-hand side of the linear system (\ref{eq:equilibration_conditions_local}) is in the range of the operator
$\cL_{h,z}$ if it is orthogonal to $\bR_{h,z}^\perp$, the null space of $\cL_{h,z}^\ast$.
Using Proposition \ref{prop-adjoint_null_space_characterization} this is obviously the case for patches $\omega_z$ with
$| \partial \omega_z \cap \Gamma_D | > 0$ since $\bR_{h,z}^\perp$ only contains zero in that case. In the case of interior
patches $\omega_z$ in the sense that $| \partial \omega_z \cap \Gamma_D | = 0$, we may insert the representation of $\bR_{h,z}^\perp$
into the right-hand side of (\ref{eq:equilibration_conditions_local}). This leads to
\begin{equation}
  \begin{split}
    ( & (\bff + \div \: \widehat{\bP}_h (\bu_h)) \phi_z , \bz_{h,z} )_{\omega_z,h}
    - \sum_{S \in \cS_{z,h}} \langle \llbracket \widehat{\bP}_h (\bu_h) \cdot \bn \rrbracket_S \phi_z , \bs_{h,z} \rangle_S
    + ( \widehat{\bP}_h (\bu_h) \phi_z , \bJ (\bgamma_{h,z}) \bF (\bu_h) )_{\omega_z} \\
    & = ( (\bff + \div \: \widehat{\bP}_h (\bu_h)) \phi_z , \cP_{h,z}^k \brho )_{\omega_z,h}
    - \sum_{S \in \cS_{z,h}} \langle \llbracket \widehat{\bP}_h (\bu_h) \cdot \bn \rrbracket_S \phi_z , \cP_{h,S}^k \brho \rangle_S
    + ( \widehat{\bP}_h (\bu_h) \phi_z , \bnabla \brho )_{\omega_z} \\
    & = ( \bff , \phi_z \cP_{h,z}^k \brho )_{\omega_z} - ( \widehat{\bP}_h (\bu_h) , \bnabla (\phi_z \cP_{h,z}^k \brho) )_{\omega_z}
    + \sum_{S \in \cS_{z,h}} \langle \llbracket \widehat{\bP}_h (\bu_h) \cdot \bn \rrbracket_S , \phi_z (\cP_{h,z}^k - \cP_{h,S}^k) \brho \rangle_S
    + ( \widehat{\bP}_h (\bu_h) , \phi_z \bnabla \brho )_{\omega_z}
  \end{split}
  \label{eq:consistency_right_hand_side}
\end{equation}
for all $( \bz_{h,z} , \bs_{h,z} , \bgamma_{h,z} ) \in \bR_{h,z}^\perp$. The first two terms vanish since
\[
  ( \widehat{\bP}_h (\bu_h)) , \bnabla (\phi_z \cP_{h,z}^{k,0} \brho) )_{\omega_z}
  = ( \bP_h (\bu_h)) , \bnabla (\phi_z \cP_{h,z}^k \brho) )_{\omega_z} = ( \bff , \phi_z \cP_{h,z}^k \brho )_{\omega_z}
\]
holds due to the definition of $\widehat{\bP}_h (\bu_h)$ as projection onto piecewise polynomials of degree $k$ and the Galerkin condition
(\ref{eq:Galerkin_condition}) which holds for piecewise polynomials of degree $k+1$ (of which $\phi_z \cP_{h,z}^k \brho$ is a fine specimen).
Therefore, for each $( \bz_{h,z} , \bs_{h,z} , \bgamma_{h,z} ) \in \bR_{h,z}^\perp$, we end up with the expression
\begin{equation}
  \begin{split}
    ( (\bff + \div \: \widehat{\bP}_h (\bu_h)) \phi_z , \bz_{h,z} )_{\omega_z,h}
    & - \sum_{S \in \cS_{z,h}} \langle \llbracket \widehat{\bP}_h (\bu_h) \cdot \bn \rrbracket_S \phi_z , \bs_{h,z} \rangle_S
    + ( \widehat{\bP}_h (\bu_h) \phi_z , \bJ (\bgamma_{h,z}) \bF (\bu_h) )_{\omega_z} \\
    & = \sum_{S \in \cS_{z,h}} \langle \llbracket \widehat{\bP}_h (\bu_h) \cdot \bn \rrbracket_S , \phi_z (\cP_{h,z}^k - \cP_{h,S}^k) \brho \rangle_S
    + ( \widehat{\bP}_h (\bu_h) , \phi_z \bnabla \brho )_{\omega_z}
  \end{split}
  \label{eq:consistency_right_hand_side_linear}
\end{equation}
for the inconsistency of the right-hand side in (\ref{eq:equilibration_conditions_local}). This motivates the choice of a modified test space such that
the term in (\ref{eq:consistency_right_hand_side_linear}) actually vanishes.

\section{A Modification Leading to Equilibrated Stresses}

\label{sec-remedy}

Our construction so far is based on using the simple component-wise $L^2 (\Omega)$-projection $\widehat{\bP}_h (\bu_h) = \cP_h^k \bP (\bu_h)$
onto the space of piecewise polynomials of degree $k$.
Due to the incompatibility of the right-hand sides in the local equilibration systems (\ref{eq:equilibration_conditions_local}) on interior vertex patches,
these problems do not possess a solution, in general. It is certainly possible to solve these systems in a least-squares sense but that would mean
that we do not get equilibrated stresses from this procedure. In particular, this means that momentum
conservation would not be satisfied locally on each element. We will therefore take up our findings from Section \ref{sec-solvability} and derive a
modification of $\widehat{\bP}_h (\bu_h)$ such that the right-hand side in (\ref{eq:equilibration_conditions_local}) becomes compatible.
In view of (\ref{eq:consistency_right_hand_side}), it is reasonable to choose the test space $\bZ_{h,z}$ as well as the test functions
$\bzeta$ on the sides $S \in \cS_{z,h}$ in such a way that they contain the rigid body modes $\brho \in \bR\bM (\bu_h)$ of the deformed configuration.
With this choice, $\cP_{h,z}^k \brho = \cP_{h,S}^k \brho = \brho$ and the sum over the sides in (\ref{eq:consistency_right_hand_side}) vanishes.
A straightforward way to do this consists in building the test spaces on the basis of piecewise polynomials in the deformed variables
$\bvarphi (\bx) = \bx + \bu_h (\bx)$ instead of $\bx$. This choice also makes sense in view of the fact that the quantities
$\div \: \bP_h^R$ and $\llbracket \bP_h^R \cdot \bn \rrbracket$ which are actually tested are mappings from the reference configuration to
(forces in) the current configuration. In fact, this modification of the test spaces is only needed for the subspace of polynomials of
degree 1 and one can use a hierarchical construction where the enrichment to polynomials of higher degree is again based on the reference
coordinates from $\bx$.

Let us assume, for the moment, that also the test space in the Galerkin formulation (\ref{eq:Galerkin_condition}) would contain the rigid body modes of
the deformed configuration. Then, the compatibility condition in (\ref{eq:consistency_right_hand_side}) would turn into
\begin{equation}
  \begin{split}
    ( \bff , \phi_z \brho )_{\omega_z} & - ( \widehat{\bP}_h (\bu_h) , \bnabla (\phi_z \brho) )_{\omega_z}
    + ( \widehat{\bP}_h (\bu_h) , \phi_z \bnabla \brho )_{\omega_z} \\
    & = ( \bff , \phi_z \brho )_{\omega_z} - ( \widehat{\bP}_h (\bu_h) , \brho \: \nabla \phi_z )_{\omega_z} \\
    & = ( \bff , \phi_z \brho )_{\omega_z} - ( \bP (\bu_h) , \brho \: \nabla \phi_z )_{\omega_z} \\
    & = ( \bff , \phi_z \brho )_{\omega_z} - ( \bP (\bu_h) , \brho \: \nabla \phi_z + \phi_z \bnabla \brho )_{\omega_z} \\
    & = ( \bff , \phi_z \brho )_{\omega_z} - ( \bP (\bu_h) , \bnabla (\phi_z \brho) )_{\omega_z} \: ,
  \end{split}
  \label{eq:consistency_miracle}
\end{equation}
if $\widehat{\bP}_h (\bu_h)$ is defined as the $L^2 (\omega_z)$-orthogonal projection with respect to piecewise polynomials in the deformed
coordinates. The compatibility term in (\ref{eq:consistency_miracle}) does indeed miraculously cancel out, if $\phi_z \brho$ is assumed
to be in the test space of the Galerkin formulation (\ref{eq:Galerkin_condition}). Using such a test space is not as far-fetched as one might think.
It would ensure invariance with respect to the rigid body modes in the deformed configuration which is not fulfilled for the use of standard
polynomial-based finite elements. However, such an approach is expected to be too complicated for practical use and therefore we need to
come up with a suitable choice for $\widehat{\bP} (\bu_h)$ leading to a compatible right-hand side in the absence of this ideal situation.

We restrict our construction to the lowest-order case $k = 1$ and consider the following slightly more general formulation of
(\ref{eq:equilibration_conditions_local}):
\begin{equation}
  \begin{split}
    ( \div \: \bP_{h,z}^\Delta , \bz_{h,z} )_{\omega_z,h}
    & = - ( ( \widehat{\bff} + \div \: \widehat{\bP}_h (\bu_h) ) \phi_z , \bz_{h,z}  )_{\omega_z,h}
    \mbox{ for all } \bz_{h,z} \in \bZ_{h,z} \: , \\
    \langle \llbracket \bP_{h,z}^\Delta \cdot \bn \rrbracket_S , \bzeta \rangle_S
    & = - \langle \llbracket \widehat{\bP}_h (\bu_h) \cdot \bn \rrbracket_S \, \phi_z , \bzeta \rangle_S
    \hspace{0.85cm} \mbox{ for all } \bzeta \in \tilde{P}_1 (S)^d \: , \: S \in \cS_{h,z} \: , \\
    ( \bP_{h,z}^\Delta \bF (\bu_h)^T , \bJ (\bgamma_{h,z}) )_{\omega_z} &
    = - ( \widehat{\bP}_h (\bu_h) \bF (\bu_h)^T \phi_z , \bJ (\bgamma_{h,z}) )_{\omega_z}
    \hspace{0.4cm} \mbox{ for all } \bgamma_{h,z} \in \bX_{h,z}
  \end{split}
  \label{eq:equilibration_conditions_local_general}
\end{equation}
where we are still free to construct $\widehat{\bff}$ in an appropriate way from $\bff$. As test space in the first equation of
(\ref{eq:equilibration_conditions_local_general}),
\begin{equation}
  \bZ_{h,z} = \{ \left. \bz_{h,z} \right|_T = \bq \circ \bvarphi \mbox{ with } \bq \in P_1 (\varphi(T))^d \}
  \label{modified_test_space_elements}
\end{equation}
could be chosen, where $\varphi$ again denotes the mapping from the reference to the (approximated) deformed configuration given by
$\bvarphi (\bx) = \bx + \bu_h (\bx)$. The test space for the second equation in (\ref{eq:equilibration_conditions_local_general}) would then be given
component-wise by transformed polynomials of the form
\begin{equation}
  \tilde{P}_1 (S)^d = \{ \bq \circ \bvarphi : \bq \in P_1 (\varphi(S)) \} \: .
  \label{modified_test_space_sides}
\end{equation}
However, in order to make sure that the rigid body modes associated with the deformed configuration $\bR\bM (\bu_h)$ are contained in the test
space, it is sufficient to replace the original undeformed rigid body modes $\bR\bM (\bzero)$ in the piecewise polynomial test space by $\bR\bM (\bu_h)$.
The test space $\bX_{h,z}$ for the third equation in (\ref{eq:equilibration_conditions_local_general}), the weak symmetry condition, may remain
unchanged since only constant rotations appear in the compatibility conditions resulting from Proposition
\ref{prop-adjoint_null_space_characterization}. For these spaces, the compatibility condition
\begin{equation}
  ( (\widehat{\bff} + \div \: \widehat{\bP}_h (\bu_h)) \phi_z , \bz_{h,z} )_{\omega_z,h}
  - \sum_{S \in \cS_{z,h}} \langle \llbracket \widehat{\bP}_h (\bu_h) \cdot \bn \rrbracket_S \phi_z , \bs_{h,z} \rangle_S
  + ( \widehat{\bP}_h (\bu_h) \phi_z , \bJ (\bgamma_{h,z}) \bF (\bu_h) )_{\omega_z} = 0
  \label{eq:compatibility_condition}
\end{equation}
for all $( \bz_{h,z} , \bs_{h,z} , \bgamma_{h,z} ) \in \bR_{h,z}^\perp$ is therefore equivalent to
\begin{equation}
  \begin{split}
    0 = ( (\widehat{\bff} + \div \: \widehat{\bP}_h (\bu_h)) \phi_z , \brho )_{\omega_z,h}
    & - \sum_{S \in \cS_{z,h}} \langle \llbracket \widehat{\bP}_h (\bu_h) \cdot \bn \rrbracket_S \phi_z , \brho \rangle_S
    + ( \widehat{\bP}_h (\bu_h) \phi_z , \bnabla \brho )_{\omega_z} \\
    & = ( \widehat{\bff} , \phi_z \brho )_{\omega_z} - ( \widehat{\bP}_h (\bu_h) , \bnabla (\phi_z \brho) )_{\omega_z}
    + ( \widehat{\bP}_h (\bu_h) , \phi_z \bnabla \brho )_{\omega_z} \\
    & = ( \widehat{\bff} , \phi_z \brho )_{\omega_z} - ( \widehat{\bP}_h (\bu_h) , \brho \: \nabla \phi_z )_{\omega_z}
  \end{split}
  \label{eq:compatibility_condition_RM}
\end{equation}
due to (\ref{eq:consistency_right_hand_side}). Making use of the Galerkin condition (\ref{eq:Galerkin_condition}), the compatibility condition
(\ref{eq:compatibility_condition_RM}) is certainly fulfilled if, on all elements $T \in \cT_h$,
\begin{equation}
  \begin{split}
    ( \widehat{\bff} , \brho \phi_z )_T & = ( \bff , (\brho \circ \bvarphi^{-1}) \phi_z )_T
    \hspace{1.8cm} \mbox{ for all } \brho \in \bR\bM (\bu_h) \: , \: z \in \cV_h^\prime \cap T \: , \\
    ( \widehat{\bP}_h (\bu_h) , \brho \: \nabla \phi_z) )_T & = ( \bP (\bu_h) , \bnabla ((\brho \circ \bvarphi^{-1}) \phi_z) )_T
    \hspace{0.65cm} \mbox{ for all } \brho \in \bR\bM (\bu_h) \: , \: z \in \cV_h^\prime \cap T
  \end{split}
  \label{eq:compatibility_condition_polynomial}
\end{equation}
holds. Note that $\brho \circ \bvarphi^{-1} \in \bR\bM (\bzero)$ holds with the original undeformed rigid body modes.
The first relation in (\ref{eq:compatibility_condition_polynomial}) constitutes $d (d+1)^2 / 2$ conditions (9 in two dimensions, 24 in
three dimensions) and can thus be fulfilled by choosing $\widehat{\bff} \in P_2 (T)^d$. The spare degrees of freedoms may be used to
minimize $\| \widehat{\bff} - \bff \|_T$ among all $\widehat{\bff} \in P_2 (T)^d$ satisfying the constraints. The second relation in (\ref{eq:compatibility_condition_polynomial}) constitutes $d (d+1)^2 / 2 - d$ conditions (7 in two dimensions, 21 in
three dimensions) since the constant rigid body modes gives zero on both sides. These conditions can be fulfilled by
$\widehat{\bP}_h \in P_1 (T)^{d \times d}$. Again, a reasonable elimination of the spare degrees of freedoms consists in minimizing
$\| \widehat{\bP}_h - \bP_h (\bu_h) \|_T$ among all $\widehat{\bP}_h \in P_1 (T)^{d \times d}$ satisfying the constraints.

We end this section with a remark on the inf-sup stability of the system (\ref{eq:equilibration_conditions_local_general}) which follows along the
same lines as in \cite{BofBreFor:09} for the linear elasticity formulation. It is easy to see that the null space associated with the first and second equation
in (\ref{eq:equilibration_conditions_local_general})
\begin{equation}
  \bPi_{h,z}^{\Delta,0} = \{ \bQ_{h,z} \in \bPi_{h,z}^\Delta : ( \div \: \bQ_{h,z} , \bz_{h,z} )_{\omega_z,h} = 0 \mbox{ for all }
  \bz_{h,z} \in \bZ_{h,z} \: , \: \langle \llbracket \bQ_{h,z}^\Delta \cdot \bn \rrbracket_S , \bzeta \rangle_S = 0 \mbox{ for all } \bzeta \in \tilde{P}_1 (S)^d
  \: , \: S \subset \omega_z \}
\end{equation}
remains unchanged by the modification of the test spaces, i.e.,
\begin{equation}
  \begin{split}
    \bPi_{h,z}^{\Delta,0} & = \{ \bQ_{h,z} \in \bPi_{h,z}^\Delta : \div \: \bQ_{h,z}^\Delta = 0 \mbox{ for all } T \subset \omega_z \: , \:
    \langle \llbracket \bQ_{h,z}^\Delta \cdot \bn \rrbracket_S = 0 \mbox{ for all } S \subset \omega_z \} \\
    & = \{ \bcurl \: \bxi_{h,z} : \bxi_{h,z} \in \bXi_{h,z} \} \: ,
  \end{split}
\end{equation}
where $\bXi_{h,z}$ is the subspace of N\'{e}d\'{e}lec elements (of the first kind) on $\omega_z$ with vanishing tangential trace on $\partial \omega_z$.
All that is left to show for the inf-sup stability of (\ref{eq:equilibration_conditions_local_general}) is therefore that
\begin{equation}
  \beta \| \bgamma_{h,z} \|_{\omega_z} \leq \sup_{\bxi_{h,z} \in \bXi_{h,z}}
  \frac{( (\bcurl \: \bxi_{h,z}) \bF (\bu_h)^T , \bJ (\bgamma_{h,z}) )_{\omega_z}}{\| \bcurl \: \bxi_{h,z} \|_{\omega_z}} \mbox{ for all } \bgamma_{h,z} \in \bX_{h,z}
  \label{eq:inf_sup_subspace}
\end{equation}
holds with a constant $\beta > 0$. If we define $\bxi_{h,z}^\varphi : \bvarphi (\omega_z) \rightarrow \R^{d \times d}$ by
$\bxi_{h,z}^\varphi \circ \bvarphi = \bxi_{h,z} \bF (\bu_h)^{-1}$, then, according to the transformation rule of the curl operator (cf. \cite[Sect. 2.1.3]{BofBreFor:13}),
$\bxi_{h,z}^\varphi \in H (\bcurl^\varphi , \bvarphi (\omega_z))$ and
\begin{equation}
  (\bcurl^\varphi \bxi_{h,z}^\varphi) \circ \bvarphi = \frac{1}{\det \bF (\bu_h)} (\bcurl \: \bxi_{h,z}) \bF (\bu_h)^T \: ,
\end{equation}
where $\bcurl^\varphi$ denotes the curl with respect to the mapped coordinates. The inf-sup condition (\ref{eq:inf_sup_subspace}) is therefore equivalent to the
existence of a constant $\beta > 0$ such that
\begin{equation}
  \beta \| \bgamma_{h,z} \|_{\bvarphi (\omega_z)} \leq \sup_{\bxi_{h,z}^\varphi \in \bXi_{h,z}^\varphi}
  \frac{( \bcurl^\varphi \: \bxi_{h,z}^\varphi , \bJ (\bgamma_{h,z}) )_{\bvarphi_z (\omega_z)}}{\| \bcurl^\varphi \: \bxi_{h,z}^\varphi \|_{\bvarphi (\omega_z)}}
  \mbox{ for all } \bgamma_{h,z} \in \bX_{h,z}
  \label{eq:inf_sup_subspace_mapped}
\end{equation}
with the mapped N\'ed\'elec space $\bXi_{h,z}^\varphi$ holds. This is exactly the inf-sup condition for the original spaces from \cite{BofBreFor:13} in mapped
coordinates using parametric Raviart-Thomas elements \cite{BerSta:16} for the stress approximation.

The combination of the inf-sup stability of the system ((\ref{eq:equilibration_conditions_local_general}) with the fact that our right-hand side
is guaranteed to be in its range ensures that there is a correction $\bP_{z,h}^\Delta$ in the broken Raviart-Thomas space leading to an
equilibrated stress $\bP_h^R$ in the end.

\section{Improved Approximation of Surface Traction Forces}

\label{sec-surface_forces}

One of the motivations for the construction of equilibrated stresses is that this leads to approximations of the surface traction forces with an
ensured convergence rate. The divergence theorem implies that
\begin{equation}
  \langle ( \bP - \bP_h^R ) \cdot \bn , \bv \rangle_{\partial \Omega} = ( \div (\bP - \bP_h^R) , \bv ) + ( \bP - \bP_h^R , \bnabla \bv )
  \label{eq:divergence_theorem}
\end{equation}
holds for all $\bv \in H^1 (\Omega)^d$. If we assume that $(\bP - \bP_h^R) \cdot \bn = \bzero$ on $\Gamma_N$ and $\div (\bP - \bP_h^R) = \bzero$ in
$\Omega$ holds (for example, since $\bf$ and $\bg$ are piecewise constant), then (\ref{eq:divergence_theorem}) turns into
\begin{equation}
  \langle ( \bP - \bP_h^R ) \cdot \bn , \bv \rangle_{\Gamma_D} = ( \bP - \bP_h^R , \bnabla \bv ) \: .
  \label{eq:divergence_theorem_simplified}
\end{equation}
This implies that
\begin{equation}
  \| ( \bP - \bP_h^R ) \cdot \bn \|_{-1/2,\Gamma_D}
  = \sup_{\bv \in H^1 (\Omega)} \frac{\langle ( \bP - \bP_h^R ) \cdot \bn , \bv \rangle_{\Gamma_D}}{\| \bnabla \bv \|}
  = \sup_{\bv \in H^1 (\Omega)} \frac{( \bP - \bP_h^R , \bnabla \bv )}{\| \bnabla \bv \|}
  \leq \| \bP - \bP_h^R \|
  \label{eq:minus_half_estimate}
\end{equation}
is satisfied which means that the approximation of the surface traction forces, measured in the $H^{-1/2} (\Gamma)$ norm, converges at least as fast
as the stress approximation in the $L^2 (\Omega)$ norm. Since, by construction, $\| \bP_h^R - \bP_h (\bu_h) \|$ is expected to be locally an
$O (h^2)$-approximation, the term on the right-hand side in (\ref{eq:minus_half_estimate}) will converge at the same order as $\| \bP - \bP (\bu_h) \|$,
in general.

If we insert $\bv \in \bR\bM (\bu_h)$, the rigid body modes in the deformed configuration, into the numerators in the middle of (\ref{eq:minus_half_estimate}),
then
\begin{equation}
  \langle ( \bP - \bP_h^R ) \cdot \bn , \bv \rangle_{\Gamma_D} = ( \bP - \bP_h^R , \bnabla \bv )
  = ( (\bP - \bP_h^R) \bF (\bu_h)^T , (\bnabla \bv) \bF (\bu_h)^{-1} ) = 0
\end{equation}
since $\bnabla \bv \bF (\bu_h)^{-1} = \bJ (\btheta)$, which constitutes a global version of (\ref{eq:rotation_and_gradient}), and $( \bP - \bP_h^R ) \bF (\bu_h)^T$
is weakly symmetric in the sense of (\ref{eq:equilibration_conditions_local_general}).

\section{Computational Results}

\label{sec-computational}

We tested our stress equilibration procedure for the well-known Cook's membrane example with a quadrilateral geometry. The corners of
the domain are located at $( 0,0 )$, $( 0.48,0.44 )$, $( 0.48,0.6 )$ and $( 0,0.44 )$ and the boundary is divided into the left line segment
$\Gamma_D$ and the lower, right, and upper segments which together form $\Gamma_N$. Figure \ref{fig-mesh3} shows this geometry and the
triangulation $\cT_3$ which is the result of three levels of uniform refinement. The surface traction force on the right boundary segment is
$\bg = ( 0 , \gamma )^T$ with different
values $\gamma > 0$, while the upper and lower boundary parts are traction-free; the volume forces $\bff$ are set to zero. In order to test the
robustness of our approach with respect to the incompressibility, we set $\mu = 1$ and $\lambda = \infty$ in the Neo-Hookean law
(\ref{eq:Neo_Hooke_pressure}) and use the displacement-pressure approximation from (\ref{eq:Galerkin_condition_pressure}) as starting point for
our stress equilibration procedure. All our computations are for the lowest-order case $k = 1$ using the Taylor-Hood combination of finite element
spaces.

\begin{figure}
\centering
 \includegraphics[scale=0.35]{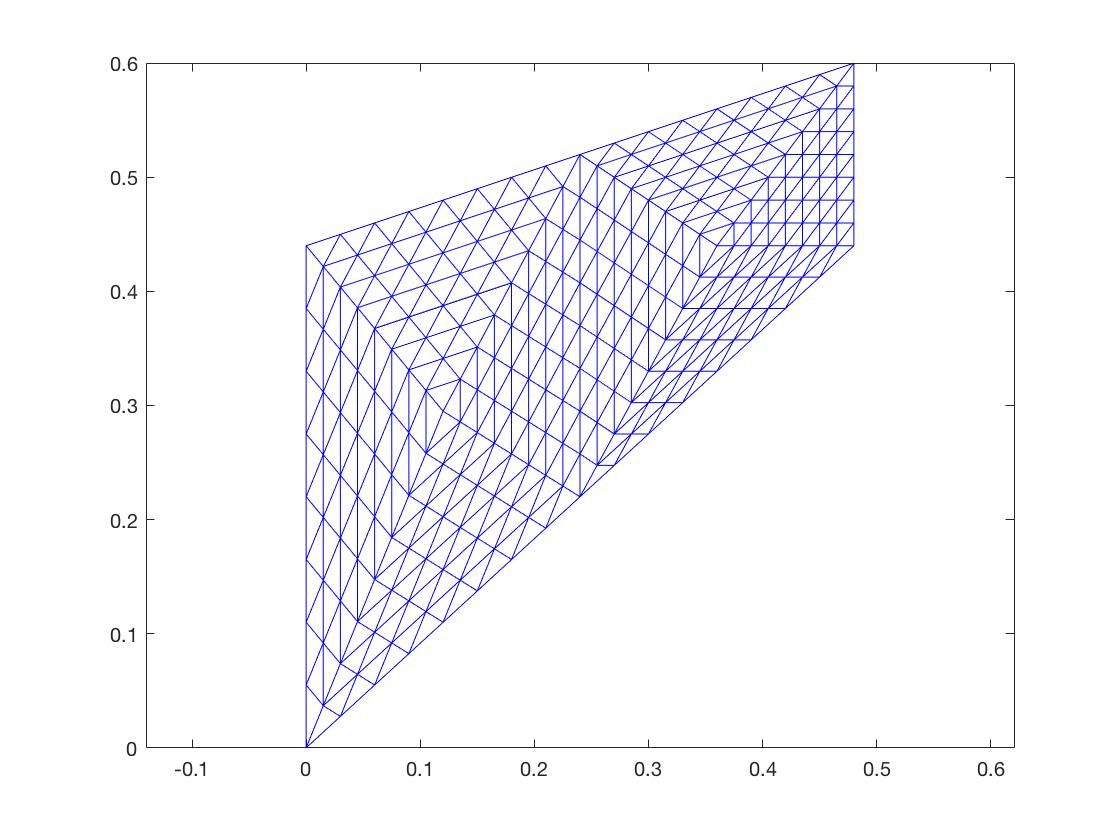}
  \caption{Cook's membrane and triangulation $\cT_3$ after three uniform refinement steps}
  \label{fig-mesh3}
\end{figure}

Of particular interest is the distribution of the traction forces on the left boundary including the singularity with infinite stress components at the upper
left corner. The distribution of the normal traction force along the left boundary is shown in Figure \ref{fig-normal_force}, for the load value $\gamma = 0.2$,
on the triangulation $\cT_5$ which results from two further uniform refinements of $\cT_3$. The left graph shows the
values for $\bn \cdot (\widehat{\bP}_h (\bu_h) \cdot \bn)$, corresponding to the projected Piola-Kirchhoff stress from the Galerkin approximation.
The right graph shows $\bn \cdot (\bP_h^R \cdot \bn)$ for the reconstructed stress. Both pictures represent piecewise affine traction force distributions
along the vertical axis. At a first glance, one may get the impression that the left distribution ``looks better than'' the right one. However, at closer inspection
it becomes obvious that the reconstructed stress in the right graph is better able to represent the singular behavior at the upper end. More importantly,
the surface forces obtained from the reconstructed Piola-Kirchhoff stress $\bP_h^R$ recover the correct resultant force
\begin{equation}
  \cI_{D,n} (\bP) := \int_{\Gamma_D} \: \bn \cdot (\bP \cdot \bn) \: ds = 0 \: .
\end{equation}
This is a consequence of the divergence theorem which implies
\begin{equation}
  \int_{\Gamma_D} \: \bP \cdot \bn \: ds = \int_\Omega \: \div \: \bP \: dx - \int_{\Gamma_N}  \: \bP \cdot \bn \: ds
  = \begin{pmatrix} 0 \\ - 0.16 \: \gamma \end{pmatrix} \: .
\end{equation}
The approximations $\cI_{D,n} (\widehat{\bP}_h (\bu_h))$ and $\cI_{D,n} (\bP_h^R)$ are shown for the two triangulations $\cT_3$ and $\cT_5$ and
several values of $\gamma$ in Tables \ref{tab-resultant_normal_force_PK} and \ref{tab-resultant_normal_force_R}. Apparently, the values produced by
$\widehat{\bP}_h$ are not exact while those coming from the stress reconstruction differ from zero only in the range of machine precision.

\begin{figure}
\centering
  \includegraphics[scale=0.22]{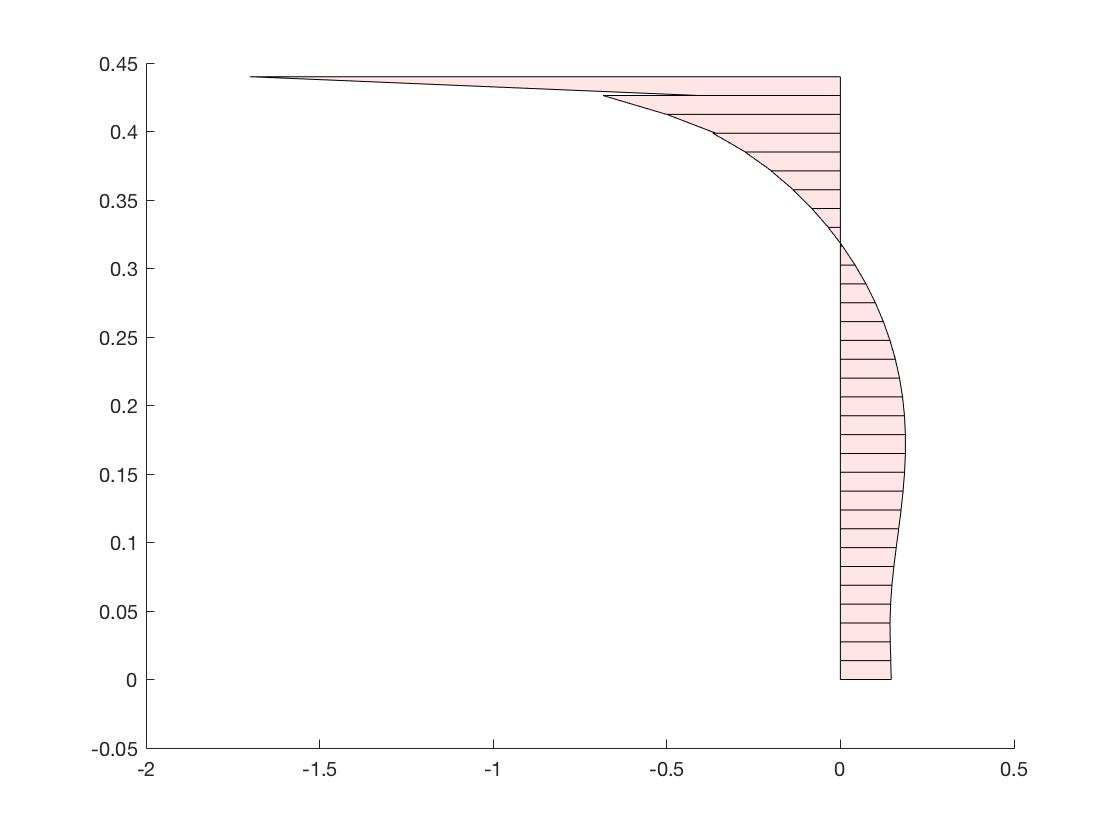}\includegraphics[scale=0.22]{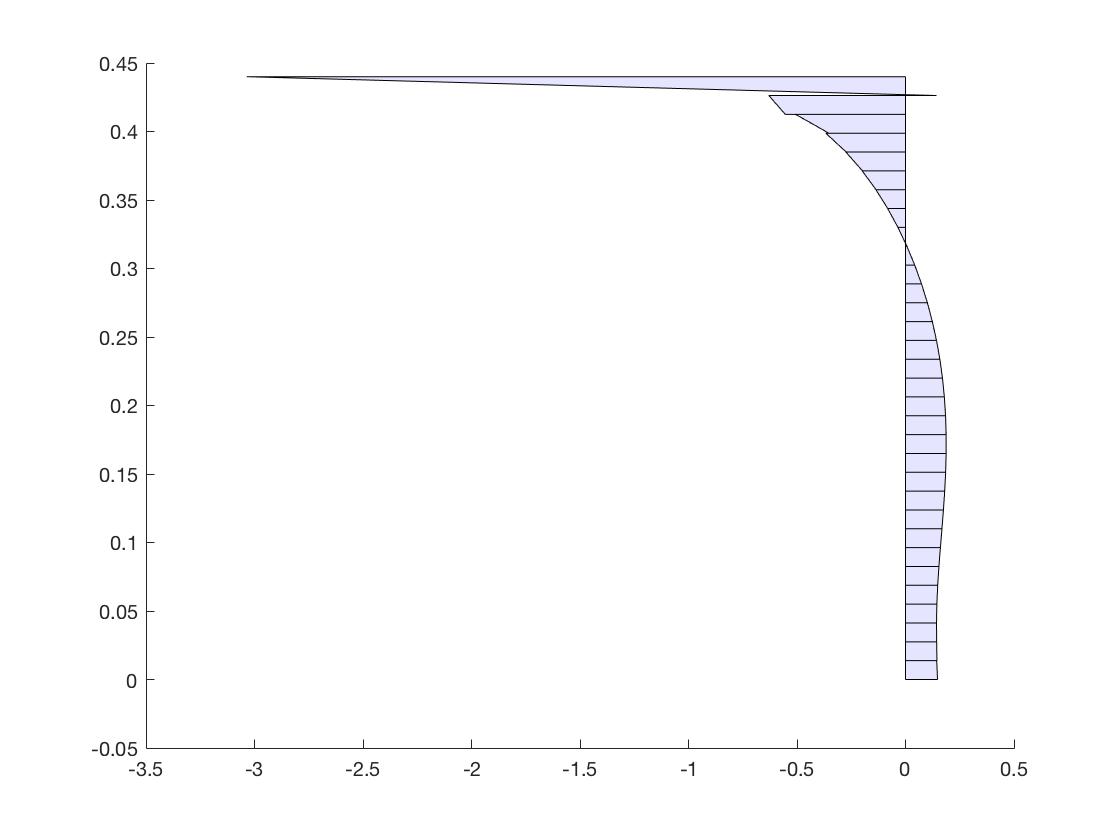}
  \caption{Normal traction forces $\bn \cdot (\widehat{\bP} (\bu_h, p_h) \cdot \bn)$ (left) and $\bn \cdot (\bP_h^R \cdot \bn)$ on $\Gamma_D$
  for $\cT_5$ ($\gamma = 0.2$)}
  \label{fig-normal_force}
\end{figure}




\begin{center}
\begin{table}
  \centering
  \begin{tabular}{|l|ccc|}
    $\cI_{D,n} (\widehat{\bP}_h (\bu_h))$ & $\gamma = 0.05$ & $\gamma = 0.2$ & $\gamma = 0.5$ \\ \hline
    $\cT_3$ & $1.69 \cdot 10^{-3}$ & $8.29 \cdot 10^{-3}$ & $2.31 \cdot 10^{-2}$ \\
    $\cT_5$ & $9.59 \cdot 10^{-4}$ & $5.40 \cdot 10^{-3}$ & $2.59 \cdot 10^{-3}$ \\ \hline
  \end{tabular}
  \caption{Approximated resultant normal traction force for $\widehat{\bP}_h (\bu_h)$}
  \label{tab-resultant_normal_force_PK}
\end{table}
\end{center}

\begin{center}
\begin{table}
  \centering
  \begin{tabular}{|l|ccc|}
    $\cI_{D,n} (\bP_h^R)$ & $\gamma = 0.05$ & $\gamma = 0.2$ & $\gamma = 0.5$ \\ \hline
    $\cT_3$ & $1.74 \cdot 10^{-11}$ & -$3.80 \cdot 10^{-9}$ & -$2.87 \cdot 10^{-10}$ \\
    $\cT_5$ & $8.11 \cdot 10^{-11}$ & $5.96 \cdot 10^{-10}$ & $2.92 \cdot 10^{-9}$ \\ \hline
  \end{tabular}
  \caption{Approximated resultant normal traction force for $\bP_h^R$}
  \label{tab-resultant_normal_force_R}
\end{table}
\end{center}
 



\begin{figure}
\centering
  \includegraphics[scale=0.35]{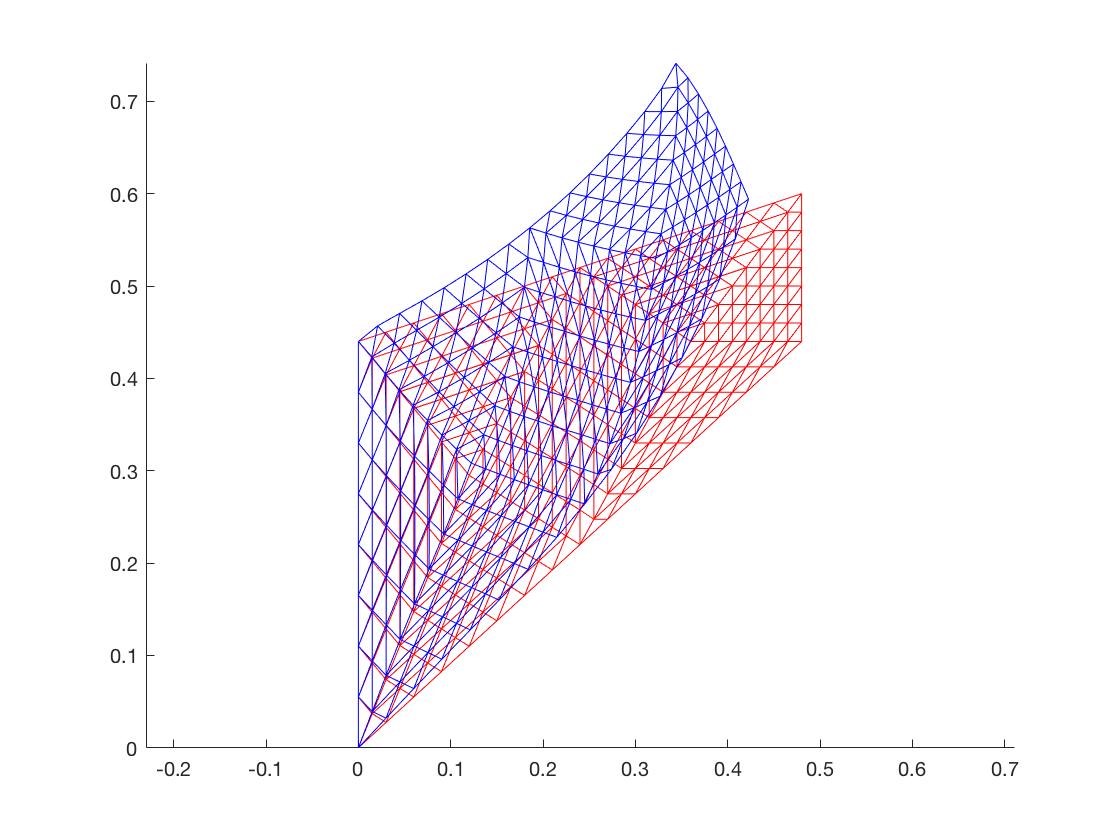}
  \caption{Reference and deformed configuration for $\gamma = 0.2$}
  \label{fig-reference_and_deformed}
\end{figure}



The reference and the deformed configuration are shown in Figure \ref{fig-reference_and_deformed} for $\gamma = 0.2$. The picture clearly
indicates that this example is well inside the geometrically nonlinear regime. Table \ref{tab-h_minus_half} compares the convergence of
$\| (\bP - \bP_h^R) \cdot \bn \|_{-1/2,\Gamma_D}$ versus $\| (\bP - \bP (\bu_h)) \cdot \bn \|_{-1/2,\Gamma_D}$ on a sequence of meshes. 
Since we do not know the exact values of $\bP \cdot \bn$ on $\Gamma_D$, we access the convergence behavior by the computation of
$\| (\bP_h^R - \bP_{2h}^R) \cdot \bn \|_{-1/2,\Gamma_D}$ and $\| (\bP (\bu_h) - \bP (\bu_{2h})) \cdot \bn \|_{-1/2,\Gamma_D}$, respectively.
The norm is evaluated approximately by
\begin{equation}
  \| s_h \|_{-1/2,\Gamma} = \sup_{v \in H^1 (\Omega)} \frac{\langle s_h , v \rangle_\Gamma}{\| v \|_{1/2,\Gamma}}
  \approx \sup_{v_h \in V_h^\ast} \frac{\langle s_h , v_h \rangle_\Gamma}{\| v_h \|_{1/2,\Gamma}} \: ,
  \label{eq:h_minus_half_evaluation}
\end{equation}
where $V_h$ denotes the space of continuous piecewise linear functions on $\cT_h$. The values in Table \ref{tab-h_minus_half} indicate that
the convergence for the equilibrated stresses is quite a bit faster than the $O (h^\alpha)$-behavior with $\alpha \approx 0.544$ expected from the
regularity of the problem. It can also be seen that the convergence rate is much higher than the one obtained for the original stresses.

\begin{center}
\begin{table}
  \centering
  \begin{tabular}{|l|ccc|}
    $\gamma = 0.2$ & $\cT_3$ & $\cT_4$ & $\cT_5$ \\ \hline
    $\| (\bP_{h}^R - \bP_{2h}^R) \cdot \bn \|_{-1/2,\Gamma_D}$ & $4.4075 \cdot 10^{-3}$ & $2.3382 \cdot 10^{-3}$ & $1.2470 \cdot 10^{-3}$ \\
    rate $\alpha$ & & 0.915 & 0.907 \\
    $\| (\bP (\bu_h) - \bP (\bu_{2h}) \cdot \bn \|_{-1/2,\Gamma_D}$ & $2.7801 \cdot 10^{-3}$ & $2.4555 \cdot 10^{-3}$ & $2.1781 \cdot 10^{-3}$ \\
    rate $\alpha$ & & 0.179 & 0.173 \\ \hline
  \end{tabular}
  \caption{Approximated resultant normal traction force for $\bP_h^R$}
  \label{tab-h_minus_half}
\end{table}
\end{center}






\section{Conclusions}

\label{sec-conclusions}

In this paper, a stress equilibration procedure for hyperelastic material models was proposed and investigated. It is necessarily based on a weakly
symmetric stress formulation and treats geometrically and materially nonlinear elasticity problems. Our main contribution is the identification of the
subspace of test functions perpendicular to the range of the equilibration system on vertex patches not attached to the Dirichlet boundary.
This result is then used to propose an appropriate projection for the Piola-Kirchhoff stress in order to get compatible patch problems. For the
moment, this stress equilibration procedure is used for its own sake, for example, in order to obtain better approximations of traction forces.
Our future goal will be to develop an a posteriori error estimator on the basis of stress equilibration for hyperelastic material models. Clearly,
this will only be possible under restrictive assumptions excluding all the known situations where uniqueness of the solution does not hold.




\bibliography{../../../../biblio/articles,../../../../biblio/books}

\providecommand{\url}[1]{\texttt{#1}}
\providecommand{\urlprefix}{}
\providecommand{\foreignlanguage}[2]{#2}
\providecommand{\Capitalize}[1]{\uppercase{#1}}
\providecommand{\capitalize}[1]{\expandafter\Capitalize#1}
\providecommand{\bibliographycite}[1]{\cite{#1}}
\providecommand{\bbland}{and}
\providecommand{\bblchap}{chap.}
\providecommand{\bblchapter}{chapter}
\providecommand{\bbletal}{et~al.}
\providecommand{\bbleditors}{editors}
\providecommand{\bbleds}{eds: }
\providecommand{\bbleditor}{editor}
\providecommand{\bbled}{ed.}
\providecommand{\bbledition}{edition}
\providecommand{\bbledn}{ed.}
\providecommand{\bbleidp}{page}
\providecommand{\bbleidpp}{pages}
\providecommand{\bblerratum}{erratum}
\providecommand{\bblin}{in}
\providecommand{\bblmthesis}{Master's thesis}
\providecommand{\bblno}{no.}
\providecommand{\bblnumber}{number}
\providecommand{\bblof}{of}
\providecommand{\bblpage}{page}
\providecommand{\bblpages}{pages}
\providecommand{\bblp}{p}
\providecommand{\bblphdthesis}{Ph.D. thesis}
\providecommand{\bblpp}{pp}
\providecommand{\bbltechrep}{}
\providecommand{\bbltechreport}{Technical Report}
\providecommand{\bblvolume}{volume}
\providecommand{\bblvol}{Vol.}
\providecommand{\bbljan}{January}
\providecommand{\bblfeb}{February}
\providecommand{\bblmar}{March}
\providecommand{\bblapr}{April}
\providecommand{\bblmay}{May}
\providecommand{\bbljun}{June}
\providecommand{\bbljul}{July}
\providecommand{\bblaug}{August}
\providecommand{\bblsep}{September}
\providecommand{\bbloct}{October}
\providecommand{\bblnov}{November}
\providecommand{\bbldec}{December}
\providecommand{\bblfirst}{First}
\providecommand{\bblfirsto}{1st}
\providecommand{\bblsecond}{Second}
\providecommand{\bblsecondo}{2nd}
\providecommand{\bblthird}{Third}
\providecommand{\bblthirdo}{3rd}
\providecommand{\bblfourth}{Fourth}
\providecommand{\bblfourtho}{4th}
\providecommand{\bblfifth}{Fifth}
\providecommand{\bblfiftho}{5th}
\providecommand{\bblst}{st}
\providecommand{\bblnd}{nd}
\providecommand{\bblrd}{rd}
\providecommand{\bblth}{th}
\begin{thebibliography}{10}

\bibitem{BerKobMolSta:19}
F.~Bertrand, B.~Kober, M.~Moldenhauer, G.~Starke, {\it Submitted for
  Publication} \textbf{2019}, arXiv: 1808.02655.

\bibitem{LubSchWei:14}
L.~Lubkoll, A.~Schiela, M.~Weiser, {\it SIAM J. Control Optim.} \textbf{2014},
  {\it 52}, 1403--1422.

\bibitem{MarHug:83}
J.~E. Marsden, T.~J.~R. Hughes, {\it Mathematical Foundations of Elasticity},
  Prentice Hall, Englewood Cliffs, \textbf{1983}.

\bibitem{Cia:88}
P.~G. Ciarlet, {\it Mathematical Elasticity Volume I: Three--Dimensional
  Elasticity}, North-Holland, Amsterdam, \textbf{1988}.

\bibitem{LeT:94}
P.~LeTallec, {\it Numerical Methods for Nonlinear Three-Dimensional
  Elasticity}, \textbf{1994}, Handb. {N}umer. {A}nal. {III}, P.G. Ciarlet and
  J. L. Lions eds., North-Holland, Amsterdam, pp. 465--662.

\bibitem{AurBeiLovReaTayWri:13}
F.~Auricchio, L.~{Beir{\~a}o da Veiga}, C.~Lovadina, A.~Reali, R.~Taylor,
  P.~Wriggers, {\it Comput. Mech.} \textbf{2013}, {\it 52}, 1153--1167.

\bibitem{CarDol:04}
C.~Carstensen, G.~Dolzmann, {\it Numer. Math.} \textbf{2004}, {\it 97}, 67--80.

\bibitem{MueStaSchSch:14}
B.~M{\"u}ller, G.~Starke, A.~Schwarz, J.~Schr{\"o}der, {\it SIAM J. Sci.
  Comput.} \textbf{2014}, {\it 36}, B795--B816.

\bibitem{BofBreFor:13}
D.~Boffi, F.~Brezzi, M.~Fortin, {\it Mixed Finite Element Methods and
  Applications}, Springer, Heidelberg, \textbf{2013}.

\bibitem{PraSyn:47}
W.~Prager, J.~L. Synge, {\it Quart. Appl. Math.} \textbf{1947}, {\it 5},
  241--269.

\bibitem{Bra:07}
D.~Braess, {\it Finite Elements: Theory, Fast Solvers, and Applications in
  Solid Mechanics} {\it \bblthirdo{} \bbledn{}}, Cambridge University Press,
  Cambridge, \textbf{2007}.

\bibitem{LadLeg:83}
P.~Ladev\`{e}ze, D.~Leguillon, {\it SIAM J. Numer. Anal.} \textbf{1983}, {\it
  20}, 485--509.

\bibitem{AinOde:93}
M.~Ainsworth, J.~T. Oden, {\it Numer. Math.} \textbf{1993}, {\it 65}, 23--50.

\bibitem{LucWoh:04}
R.~Luce, B.~Wohlmuth, {\it SIAM J. Numer. Anal.} \textbf{2004}, {\it 42},
  1394--1414.

\bibitem{CaiZha:12a}
Z.~Cai, S.~Zhang, {\it SIAM J. Numer. Anal.} \textbf{2012}, {\it 50}, 151--170.

\bibitem{ErnVoh:15}
A.~Ern, M.~Vohral{\'{\i}}k, {\it SIAM J. Numer. Anal.} \textbf{2015}, {\it 53},
  1058--1081.

\bibitem{BraSch:08}
D.~Braess, J.~Sch{\"o}berl, {\it Math. Comp.} \textbf{2008}, {\it 77},
  651--672.

\bibitem{ParBonHuePer:06}
N.~Par{\'e}s, J.~Bonet, A.~Huerta, J.~Peraire, {\it Comput. Methods Appl. Mech.
  Engrg.} \textbf{2006}, {\it 195}, 406--429.

\bibitem{NicWitWoh:08}
S.~Nicaise, K.~Witowski, B.~Wohlmuth, {\it IMA J. Numer. Anal.} \textbf{2008},
  {\it 28}, 331--353.

\bibitem{Kim:11a}
K.-Y. Kim, {\it J. KSIAM} \textbf{2011}, {\it 16}, 1--13.

\bibitem{Kim:11b}
K.-Y. Kim, {\it SIAM J. Numer. Anal.} \textbf{2011}, {\it 49}, 2364--2385.

\bibitem{AinAllBarRan:12}
M.~Ainsworth, A.~Allendes, G.~R. Barrenechea, R.~Rankin, {\it IMA J. Numer.
  Anal.} \textbf{2012}, {\it 32}, 417--447.

\bibitem{HanSteVoh:12}
A.~Hannukainen, R.~Stenberg, M.~Vohral{\'{\i}}k, {\it Numer. Math.}
  \textbf{2012}, {\it 122}, 725--769.

\bibitem{BerMolSta:19a}
F.~Bertrand, M.~Moldenhauer, G.~Starke, {\it Comput. Meth. Appl. Math.}
  \textbf{2019}, To Appear in Print, DOI:10.1515/cmam-2018-0004.

\bibitem{BotRie:19}
M.~Botti, R.~Riedlbeck, {\it Comput. Meth. Appl. Math.} \textbf{2019}, To
  Appear in Print, DOI:10.1515/cmam-2018-0012.

\bibitem{HauHec:13}
P.~Hauret, F.~Hecht, {\it SIAM J. Sci. Comput.} \textbf{2013}, {\it 35},
  B291--B314.

\bibitem{BofBreFor:09}
D.~Boffi, F.~Brezzi, M.~Fortin, {\it Commun. Pure Appl. Anal.} \textbf{2009},
  {\it 8}, 95--121.

\bibitem{BerSta:16}
F.~Bertrand, G.~Starke, {\it SIAM J. Numer. Anal.} \textbf{2016}, {\it 54},
  3648Ñ--3667.

\end{thebibliography}



\end{document}